\documentclass[11pt]{amsart}

\usepackage{amsmath,amssymb,amsthm,amscd,mathrsfs}

\usepackage{geometry}
\usepackage{graphicx}
\usepackage{subfig}
\usepackage{tikz}
\usepackage{extarrows} 

\usepackage{enumitem}
\usepackage{url,verbatim}
\usepackage[normalem]{ulem}
\usepackage{fixme}
\usepackage{centernot}
\usepackage{xcolor}
\RequirePackage[colorlinks,citecolor=blue,urlcolor=blue]{hyperref}

\makeatletter
\newcommand{\xnleftrightarrow}[2][]{%
  \mathrel{\ooalign{%
    $\xleftrightarrow[#1]{#2}$\cr
\hidewidth$\mkern2mu\not\phantom{\xleftrightarrow[#1]{#2}}$\hidewidth\cr
  }}}
\makeatother
\DeclareMathOperator{\addr}{addr}

\calclayout

\usepackage[colorlinks,citecolor=blue,urlcolor=blue]{hyperref}

\theoremstyle{plain}
\newtheorem{theorem}{Theorem}[section]
\newtheorem{definition}[theorem]{Definition}
\newtheorem{lemma}[theorem]{Lemma}
\newtheorem{proposition}[theorem]{Proposition}
\newtheorem{corollary}[theorem]{Corollary}
\newtheorem{example}[theorem]{Example}
\newtheorem{assumption}[theorem]{Assumption}
\newtheorem{remark}[theorem]{Remark}

\newtheorem{construction}[theorem]{Construction}
\newtheorem{conjecture}[theorem]{Conjecture}

\numberwithin{equation}{section}
\numberwithin{figure}{section}

\newcommand\ol{\overline}

\newcommand\RR{{\mathbb R}}

\newcommand\PP{{\mathbb P}}
\newcommand\HH{{\mathbb H}}

\newcommand\pcs{{p_c^{\mathrm{site}}}}

\newcommand\si{\sigma}
\newcommand\eps{\epsilon}

\renewcommand\ell{l}

\title{Planar Site Percolation, End Structure, and the Benjamini-Schramm Conjecture}

\author{Zhongyang Li}
\address{Department of Mathematics,
University of Connecticut,
Storrs, Connecticut 06269-3009, USA}
\email{zhongyang.li@uconn.edu}
\urladdr{\url{https://mathzhongyangli.wordpress.com}}

\begin{document}
\maketitle

\begin{abstract}
Let $G$ be an infinite, connected, locally finite planar graph and consider i.i.d.\ Bernoulli$(p)$ site percolation.
Write $p_c^{\mathrm{site}}(G)$ and $p_u^{\mathrm{site}}(G)$ for the critical and uniqueness thresholds.
Using a well--separated Freudenthal embedding $G\hookrightarrow\mathbb S^2$, we introduce a cycle--separation equivalence on ends and associated
``directional'' thresholds $p^{\mathrm{site}}_{c,F}(G)$.

When the set of end--equivalence classes is countable, we show that
$p_c^{\mathrm{site}}(G)=\inf_F p^{\mathrm{site}}_{c,F}(G)$ and that for every
$p\in\bigl(\tfrac12,\,1-p_c^{\mathrm{site}}(G)\bigr)$ there are almost surely infinitely many infinite open clusters.
Combined with the $0/\infty$ theorem of Glazman--Harel--Zelesko for $p\le \tfrac12$, this yields non--uniqueness throughout the full coexistence interval
$\bigl(p_c^{\mathrm{site}}(G),\,1-p_c^{\mathrm{site}}(G)\bigr)$, and hence $p_u^{\mathrm{site}}(G)\ge 1-p_c^{\mathrm{site}}(G)$ in this setting.
This resolves the extension problem posed by Glazman--Harel--Zelesko for the upper half of the coexistence regime under a natural countability hypothesis.

In contrast, for graphs with uncountably many end--equivalence classes we give criteria guaranteeing infinitely many infinite clusters above criticality,
and we construct an explicit locally finite planar graph of minimum degree at least $7$ for which
$p_u^{\mathrm{site}}(G)<1-p_c^{\mathrm{site}}(G)$.
Consequently, the Benjamini--Schramm conjecture (Conjecture 7 in \cite{bs96}) that planarity together with minimal vertex degree at least 7 forces infinitely many infinite clusters for all
$p\in(p_c,1-p_c)$ does not hold in full generality.

Our proofs combine a cutset characterization of $p_c^{\mathrm{site}}$ with a planar alternating--arm exploration organized by an end--adapted boundary decomposition.
\end{abstract}

\section{Introduction}

Percolation theory, introduced by Broadbent and Hammersley in the 1950s (\cite{BH57}),
is one of the simplest and most fundamental probabilistic models exhibiting phase transitions.
It captures the emergence of large-scale connectivity in random media,
and serves as a cornerstone of modern probability and statistical physics.
We refer to the surveys of Duminil-Copin~\cite{DuminilCopin2017Sixty}
and Beffara--Duminil-Copin~\cite{BeffaraDuminilCopin2013},
as well as the classical monographs by Grimmett~\cite{Grimmett1999}
and Bollob{\'a}s--Riordan~\cite{BollobasRiordan2006},
for an overview of the field and its connections to critical phenomena and universality.

\medskip

\begin{definition}[Critical and uniqueness thresholds]\label{def:pc_pu}
Let \(G = (V,E)\) be an infinite, connected graph, and consider i.i.d.\ Bernoulli site percolation on \(G\) with parameter \(p \in [0,1]\).
Each vertex \(v \in V\) is declared \emph{open} with probability \(p\) and \emph{closed} otherwise, independently of all other vertices.
Let \(\mathbb{P}_p\) denote the corresponding product measure.

\smallskip
\noindent
The \emph{critical threshold} (or \emph{critical probability}) for site percolation on \(G\) is defined by
\[
  p_c^{\mathrm{site}}(G)
  := \inf\bigl\{\, p \in [0,1] :
    \mathbb{P}_p(\text{there exists an infinite open cluster}) > 0
  \bigr\}.
\]
Equivalently, for every \(p < p_c^{\mathrm{site}}(G)\), all open clusters are \(\mathbb{P}_p\)-a.s.\ finite, whereas for every \(p > p_c^{\mathrm{site}}(G)\), there exists at least one infinite open cluster a.s., by the Kolmogorov 0--1 law for tail events.

\smallskip
\noindent
The \emph{uniqueness threshold} is defined by
\begin{equation}\label{dpu}
  p_u^{\mathrm{site}}(G)
  := \inf\bigl\{\, p \in [0,1] :
     \mathbb{P}_p(\text{there is a unique infinite open cluster}) > 0
  \bigr\}.
\end{equation}
\end{definition}

\begin{assumption}\label{ap12}
Throughout the paper, whenever we study i.i.d.~Bernoulli site percolation on a graph $G$, we assume \emph{without loss of generality} that $G$ is \textbf{simple}: it contains no self-loops (edges joining a vertex to itself), and any pair of distinct vertices are connected by at most one edge.
\end{assumption}

For planar transitive graphs with a self-dual (or matching) structure, such as the triangular lattice for site percolation, one has
\(p_c^{\mathrm{site}}=1/2\) (see, e.g., Kesten~\cite{Kesten1980,Kesten1982}). 
At criticality ($p=\frac{1}{2}$) there is almost surely no infinite cluster, by Harris's theorem~\cite{HL60}. 
Above \(p_c^{\mathrm{site}}\), the infinite cluster is a.s.\ unique for quasi-transitive amenable graphs by the Burton--Keane theorem~\cite{BurtonKeane1989}. 
Consequently, for quasi-transitive amenable graphs,
\[
  p_c^{\mathrm{site}} = p_u^{\mathrm{site}}.
\]

\medskip
In contrast, for non-transitive or irregular planar graphs,
duality and uniqueness arguments are generally unavailable.
The possibility of multiple infinite clusters in such non-symmetric settings
is subtle and largely open.
Benjamini and Schramm~\cite{bs96} initiated the systematic study of
percolation on general planar graphs and formulated the following conjectures.

\begin{conjecture}\label{c7bs}
(\emph{Conjecture~7} in \cite{bs96}).
Let \(G\) be a planar graph with minimum vertex degree at least \(7\).
Then \(p_c^{\mathrm{site}}(G) < \tfrac{1}{2}\), and for every
\(p \in \bigl(p_c^{\mathrm{site}}(G),\, 1 - p_c^{\mathrm{site}}(G)\bigr)\),
there are almost surely infinitely many infinite open clusters.
\end{conjecture}

\begin{conjecture}\label{c8bs}
(\emph{Conjecture~8} in \cite{bs96}).
Let \(G\) be a planar graph and set \(p=\tfrac{1}{2}\).
If percolation occurs almost surely (i.e., there exists an infinite open cluster a.s.),
then almost surely there are infinitely many infinite open clusters.
\end{conjecture}

These conjectures, formulated by Benjamini and Schramm, have remained open for nearly three decades.
They generalize the “zero–one–infinity’’ law proved for transitive graphs by Häggström and Peres~\cite{HaggstromPeres1999}.
They are also closely related to the bond-percolation duality relation
\(p_c^{\mathrm{bond}}(G)+p_c^{\mathrm{bond}}(G^\dagger)=1\),
which is known to hold for transitive~\cite{bsjams}, quasi-transitive~\cite{GrZL22,GrZL221}and some semi-transitive~\cite{ZL23} planar graphs,
but does not automatically extend to arbitrary planar graphs.

\medskip
Partial progress toward these conjectures has been made under various structural assumptions.
Newman and Schulman~\cite{NewmanSchulman1981}
and Gandolfi--Grimmett--Russo~\cite{GGR1988}
established uniqueness results for product measures on lattices.
Benjamini, Lyons, Peres, and Schramm~\cite{BLPS1999}
provided general criteria for uniqueness and nonuniqueness
on unimodular quasi-transitive graphs.

Subsequent work on nonamenable and hyperbolic planar graphs
revealed a wealth of new phenomena.
Lalley~\cite{SL98,SL01}
analyzed percolation on Fuchsian groups and on hyperbolic tessellations,
establishing the existence of a nontrivial interval
\((p_c,p_u)\) of coexistence where infinitely many infinite clusters appear.
 Benjamini and Schramm~\cite{BenjaminiSchramm2001}
studied percolation in the hyperbolic plane, giving sharp thresholds
and a geometric characterization of nonamenable planar graphs.
Despite these advances, the inequality
\(p_u^{\mathrm{site}} \ge 1 - p_c^{\mathrm{site}}\)
remained open even for bounded-degree planar graphs.

\medskip
Li~\cite{ZL231} introduced a \emph{tree--embedding framework}
for planar site percolation and for arbitrary locally finite graphs,
linking probabilistic cutset estimates with the critical percolation probability~\cite{ZL24}.
Building on this foundation,
Li~\cite{ZL24} developed a coupling characterization
of the critical threshold \(p_c^{\mathrm{site}}(G)\)
via subcritical cutset probabilities,
providing a new analytic approach beyond transitive or periodic settings.
This line of work follows a classical contour/cutset tradition
dating back to Broadbent--Hammersley~\cite{BH57} and Hammersley~\cite{HM57},
and it connects with the modern sharpness framework of
Duminil--Copin and Tassion~\cite{DCT16}.
Li’s results resolved conjectures of
Kahn~\cite{JK03}, Peres--Lyons~\cite{LP16}, and Tang~\cite{Tang2023}
for site percolation, demonstrating that probabilistic, analytic, and topological tools
can be unified to study planar percolation without any symmetry assumptions.

For \emph{properly embedded} planar graphs
(i.e., embeddings in the plane with no accumulation points of edges or vertices),
Haslegrave and Panagiotis~\cite{HP21} proved that \(p_c^{\mathrm{site}}<\tfrac12\).
Conjecture~\ref{c7bs}
for properly embedded planar graphs was resolved in~\cite{ZL231},
which also provided an alternative proof of \(p_c^{\mathrm{site}}<\tfrac12\).
Furthermore, Conjecture~\ref{c8bs} was proved in~\cite{GHZ25} for locally finite planar graphs:
for every \(p\in(p_c^{\mathrm{site}},\,\tfrac12]\),
almost surely there exist infinitely many infinite \(1\)--clusters,
implying \(p_u^{\mathrm{site}}\ge\tfrac12\),
where \(p_u^{\mathrm{site}}\) is the uniqueness threshold defined in~\eqref{dpu}.
On Page~3 of~\cite{bs96}, it was further \emph{expected} that
the nonuniqueness regime extends up to \(1-p_c^{\mathrm{site}}\);
the present paper discusses this expectation.

\medskip
\noindent\textbf{Main result.}
\begin{definition}[Planar graph]
A graph \(G\) is \emph{planar} if it admits an embedding in the plane \(\RR^{2}\); that is, \(G\) can be drawn
with vertices as points and edges as simple arcs that meet only at vertices.
\end{definition}

\begin{remark}
By stereographic projection, \(G\) admits such a drawing (i.e.~ with vertices as points and edges as simple arcs that meet only at vertices) on the sphere \(\mathbb{S}^{2}\) if it does on
\(\RR^{2}\). Since the hyperbolic plane \(\HH^{2}\) is homeomorphic to \(\RR^{2}\) (e.g., via the Poincaré disk model),
\(G\) also admits such a drawing on \(\HH^{2}\).
\end{remark}

\noindent\textbf{Freudenthal embedding and end-equivalence (informal).}
By Section~\ref{sect2}, every infinite, connected, locally finite planar graph admits a
well--separated embedding $\phi:G\hookrightarrow\mathbb S^2$ extending the Freudenthal compactification,
so that ends of $G$ correspond bijectively to the accumulation set $\mathcal{A}(\phi)\subset\mathbb S^2$.
Following Section~\ref{sec3}, we call two ends \emph{equivalent} if no cycle of $G$ separates
their accumulation points; we denote by $\mathcal F(G)$ the set of equivalence classes and
by $p^{\mathrm{site}}_{c,F}(G)$ the corresponding $F$--connectivity threshold.
(Precise definitions are given in Sections~\ref{sect2}-\ref{sec3}.)

Our main result is Theorem~\ref{thm:main1} below. 

\begin{theorem}\label{thm:main1}
Let $G=(V,E)$ be an infinite, connected, locally finite planar graph and consider
i.i.d.\ Bernoulli$(p)$ \emph{site} percolation on $G$.
Write $p_c^{\mathrm{site}}(G)$ and $p_u^{\mathrm{site}}(G)$ for the critical and uniqueness thresholds.

Fix a well--separated Freudenthal embedding $G\hookrightarrow \mathbb S^2$.
Let $\mathcal{A}$ be the corresponding set of ends (accumulation points), and let $\mathcal F:=\mathcal{A}/\!\sim$
be the cycle--separation equivalence classes of ends.
For each $F\in\mathcal F(G)$ let $p^{\mathrm{site}}_{c,F}(G)$ be the corresponding
$F$--connectivity threshold (Definition~\ref{df36}).
Set
\[
p_*^{\mathrm{site}}(G)\ :=\ \inf_{F\in\mathcal F(G)} p^{\mathrm{site}}_{c,F}(G).
\]

\begin{enumerate}[label=\textup{(\roman*)}]
\item \textbf{Upper-half nonuniqueness up to $1-p_*$.}
For every
\[
p\in\Bigl(\tfrac12,\;1-p_*^{\mathrm{site}}(G)\Bigr),
\]
there are $\mathbb P_p$--almost surely infinitely many infinite open clusters.
Consequently,
\[
p_u^{\mathrm{site}}(G)\ \ge\ 1-p_*^{\mathrm{site}}(G).
\]

\item \textbf{When $\mathcal F(G)$ is countable, the bound becomes $1-p_c$.}
If $\mathcal F(G)$ is countable (in particular, if $G$ has countably many ends), then
\[
p_c^{\mathrm{site}}(G)\ =\ p_*^{\mathrm{site}}(G),
\]
and hence for every
\[
p\in\bigl(p_c^{\mathrm{site}}(G),\;1-p_c^{\mathrm{site}}(G)\bigr)
\]
there are $\mathbb P_p$--a.s.\ infinitely many infinite open clusters.
In particular,
\[
p_u^{\mathrm{site}}(G)\ \ge\ 1-p_c^{\mathrm{site}}(G).
\]

\item \textbf{Uncountably many end equivalence classes}
\begin{enumerate}
\item \textbf{A sufficient condition for infinitely many clusters.}
Assume $\mathcal F$ is uncountable and let $p$ satisfy
\[
\max\Bigl\{\tfrac12,\ 1-p_*^{\mathrm{site}}(G)\Bigr\}\ \le\ p\ <\ 1-p_c^{\mathrm{site}}(G).
\]
Define the (deterministic) set
\[
A_1(p):=\Bigl\{a\in \mathcal{A}:\ \mathbb P_p\bigl(\exists v\in V:\ v \stackrel{1}{\longleftrightarrow} a\bigr)=1\Bigr\}.
\]
If
\begin{equation}\label{eq:case2a}
\mathbb P_p\Bigl(\ \bigcup_{a\in A\setminus A_1(p)}\ \bigcup_{v\in V}\{v \stackrel{1}{\longleftrightarrow} a\}\ \Bigr)=1,
\end{equation}
 then $\mathbb P_p$--almost surely there are infinitely many
infinite open clusters.

\item \textbf{$p_u\ge 1-p_c$ can be violated.}
There exists an infinite, connected, locally finite planar graph $G$ with minimum degree at least $7$
and uncountably many end--equivalence classes such that
\[
p_c^{\mathrm{site}}(G)<\tfrac12
\qquad\text{and}\qquad
p_u^{\mathrm{site}}(G)\ <\ 1-p_c^{\mathrm{site}}(G).
\]
More precisely, for this $G$ there exists $p\in(\tfrac12,\ 1-p_c^{\mathrm{site}}(G))$ such that
$\mathbb P_p(1\le N_\infty(p)<\infty)=1$, and hence (by finite energy on a connected graph)
$\mathbb P_p(N_\infty(p)=1)>0$.
\end{enumerate}
Here $N_{\infty}$ is the number of infinite 1-clusters.
\end{enumerate}

\end{theorem}

\noindent\textbf{When is $\mathcal F(G)$ countable?}
The countability assumption in Theorem~\ref{thm:main1}\,(ii) is automatic whenever $G$ has only
countably many ends, since $\mathcal F(G)$ is a partition of the end space.
In particular, Theorem~\ref{thm:main1}\,(i) applies to every countably ended planar graph.

It also applies to every planar graph admitting a \emph{proper} embedding into $\mathbb R^2$
(i.e.\ with no finite accumulation points of vertices or edges). Indeed, in a proper embedding,
the closed interior of any cycle is compact and hence meets $V(G)$ in finitely many vertices,
so no cycle can separate two ends. Consequently all ends are equivalent and $\mathcal F(G)$ is a singleton.

\medskip
\noindent\textbf{A countability threshold in planar uniformization.}
The role of countability in our main theorem parallels a classical rigidity phenomenon
in planar conformal geometry.
He and Schramm proved that every planar domain with at most countably many boundary
components is conformally equivalent to a circle domain, uniquely up to M\"obius transformations,
whereas such uniqueness may fail in the presence of uncountably many boundary components
\cite{HeSchramm93}.
In a similar spirit, we show that the inequality $p_u^{\mathrm{site}}\ge 1-p_c^{\mathrm{site}}$
holds under a countability hypothesis on end--equivalence classes, and we construct a
locally finite planar counterexample once this hypothesis is dropped.

\medskip

\noindent\textbf{Relation to existing work.}
Glazman--Harel--Zelesko \cite{GHZ25} proved a $0/\infty$ law for Bernoulli site percolation on
infinite, connected, locally finite planar graphs in the range $p\le \tfrac12$; in particular,
throughout $(p_c^{\mathrm{site}}(G),\,\tfrac12]$ there are almost surely infinitely many infinite open clusters.
They further anticipated that non--uniqueness should persist up to $1-p_c^{\mathrm{site}}(G)$
without transitivity or degree assumptions (cf.\ also \cite[Page 3]{GHZ25}).

Theorem~\ref{thm:main1}\,(ii) confirms this upper--half extension under the additional hypothesis
that the set $\mathcal F(G)$ of end--equivalence classes is countable: for every
$p\in(\tfrac12,\,1-p_c^{\mathrm{site}}(G))$ there are almost surely infinitely many infinite open clusters.
Combined with \cite{GHZ25}, this yields non--uniqueness throughout the full coexistence interval
$(p_c^{\mathrm{site}}(G),\,1-p_c^{\mathrm{site}}(G))$ for all graphs covered by the countability hypothesis,
and in particular $p_u^{\mathrm{site}}(G)\ge 1-p_c^{\mathrm{site}}(G)$ in this setting.

By contrast, Theorem~\ref{thm:main1}\,(iii) shows that such an extension cannot hold in complete generality:
we give a criterion in the uncountable--class regime guaranteeing infinitely many infinite clusters
(see Section~\ref{s6}), and we also construct a locally finite planar counterexample
(see Section~\ref{s8}) for which $p_u^{\mathrm{site}}(G)<1-p_c^{\mathrm{site}}(G)$.
In particular, this yields a counterexample to the non--uniqueness conclusion of
Benjamini--Schramm's Conjecture~7 in \cite{bs96} (Conjecture \ref{c7bs}) even under the minimum--degree condition $\delta(G)\ge 7$.

\medskip
\noindent\textbf{The FCA framework.}
This paper introduces an \emph{FCA framework} (Freudenthal–Cutset–Arms) for planar percolation, distilling our method into three structural components:
\begin{enumerate}
\item[\textbf{(F1)}] \emph{Freudenthal Embedding.}
Every locally finite planar graph admits a well – separated embedding on \(\mathbb{S}^2\), which identifies ends with accumulation points and permits the definition of disjoint arm corridors aimed at distinct ends.
\item[\textbf{(F2)}] \emph{Cutset Characterization of \(p_c^{\mathrm{site}}(G)\).}
A coupling formulation via supercritical cutset bounds yields uniform control of connectivity probabilities and an analytic handle on the critical threshold, without transitivity or symmetry assumptions.
\item[\textbf{(F3)}] \emph{Alternating–Arm Exploration.}
A multi–arm scheme together with boundary–layer explorations enforces conditional independence across separated regions, forcing infinitely many disjoint infinite clusters throughout \((p_c^{\mathrm{site}},\,1-p_c^{\mathrm{site}})\).
\end{enumerate}

\noindent
Implemented in full for Bernoulli site percolation, the FCA framework yields the non-uniqueness result in the entire coexistence region for all infinite, connected, locally finite planar graphs  with countably many end equivalent classes (Theorem \ref{thm:main1}(ii)). Its structure appears flexible and may extend—once analogues of \textup{(F2)}–\textup{(F3)} are verified—to other planar statistical–mechanical models, such as the bond percolation model and the Ising model.

Beyond planarity, the FCA framework yields quantitative supercritical bounds for Bernoulli
\emph{site} percolation.
For any infinite, connected, locally finite graph $G=(V,E)$, any parameter $p>p^{\mathrm{site}}_{c}(G)$,
and any (finite or infinite) set of vertices $S\subset V$, in \cite{ZL262} we derive explicit exponential-type
upper bounds on the disconnection probability $\mathbb{P}_{p}(S\nleftrightarrow\infty)$.
The estimates are expressed in terms of a packing profile of $S$, encoded by a
$(p,\varepsilon,c)$--packing number,  which counts how many well-separated vertices in $S$
exhibit controlled local-to-global connectivity.
The proof combines a local functional characterization of $p^{\mathrm{site}}_{c}$ from \cite{ZL24}
with a packing construction and an amplification-by-independence argument, providing
progress toward Problem~1.6 of \cite{DC20}.

\medskip
We also exhibit a planar but non–locally–finite example
to demonstrate that local finiteness is indispensable.
This completes the resolution of the planar coexistence problem
for Bernoulli site percolation; see Section \ref{s82}.

\medskip
\noindent\textbf{Organization of the paper.}
In Section~\ref{sect2} we recall the Freudenthal compactification of a locally finite graph
and establish the existence of a well--separated planar embedding $G\hookrightarrow\mathbb S^2$
for every infinite, connected, locally finite planar graph.  This identifies the space of ends
with the accumulation set of the embedding and provides the topological language needed
to aim arms toward prescribed ends.

In Section~\ref{sec3} we introduce the cycle--separation equivalence on the accumulation set,
define the end--equivalence classes $\mathcal F(G)$, and attach to each $F\in\mathcal F(G)$ the
$F$--connectivity threshold $p^{\mathrm{site}}_{c,F}(G)$.  We also record the basic dichotomy between
the countable and uncountable regimes for $\mathcal F(G)$ that will guide the remainder of the paper.

Section~\ref{sec4} develops the geometric input for the exploration argument.
Starting from a finite connected vertex set $S$ and an end class $F$, we define an $F$--adapted
boundary cycle $\partial_F S$ and formulate alternating arm events across $\partial_F S$.

In Section~\ref{sec5} we prove the analytic estimates driving the supercritical argument:
a cutset--based differential inequality for $F$--connectivity, its reformulation in terms of the
$\phi$--functional, and a collection of uniform connectivity bounds stable under truncations
away from other ends.  As a consequence, when $\mathcal F(G)$ is countable we obtain
$p^{\mathrm{site}}_c(G)=\inf_{F\in\mathcal F(G)}p^{\mathrm{site}}_{c,F}(G)$.

Section~\ref{s5} combines the above ingredients in an end--targeted multi--arm
exploration to force conditional independence across separated regions, yielding infinitely many
infinite clusters throughout the relevant part of the coexistence regime.

In Section~\ref{s6} we treat graphs with uncountably many end--equivalence classes
in the positive regime (Case~(2a)), deriving criteria that still guarantee the existence of infinitely
many infinite clusters.

Finally, Section~\ref{s8} contains sharpness examples.  We construct a locally finite planar
graph in the uncountable regime for which $p^{\mathrm{site}}_u(G)<1-p^{\mathrm{site}}_c(G)$ (Case~(2b)),
and we also exhibit a planar but non--locally finite example showing that local finiteness is indispensable
for coexistence statements of Benjamini--Schramm type.

\section{Ends, Freudenthal compactification, and well--separated embeddings}
\label{sect2}

The purpose of this section is to put locally finite planar graphs into a canonical topological
setting.  We recall the Freudenthal compactification $|G|$ of a connected locally finite graph and
use classical embeddability results to obtain a well--separated embedding $G\hookrightarrow\mathbb S^2$
in which the ends of $G$ correspond bijectively to the accumulation points of the embedded image.
We also record two consequences that will be used repeatedly later: any infinite connected subgraph
determines an end together with a nested family of infinite subgraphs inside the corresponding
components, and under the embedding this yields an accumulation point at which every neighborhood
contains an infinite connected subgraph.

\begin{definition}\label{df21}
Let $G=(V,E)$ be a graph. An end of G is a map $\omega$ that assigns to every finite set
of vertices $K\subset V$ a connected component $\omega(K)$ of $G\setminus K$, and satisfies the consistency
condition 
$\omega(K)\subseteq \omega(K_0)$ whenever $K_0\subseteq K$.

It is not hard to see that if $K \subset V$ is finite and $F$ is an infinite component of $G \setminus K$, then
there is an end of $G$ satisfying $e(K)=F$.  The collection of
sets of the form $\{e : e(K) = F\}$ is a sub-basis for a topology on the set of ends of a connected
graph. 

Let $\Omega(G)$ be the set consisting of all the ends of $G$.
\end{definition}

We shall first define a topological space $|G|$ associated with a graph $G=(V,E,\Omega)$ and its ends. See also Sect.~8.6 of \cite{Diestel2017}. 

We start with the set $V\cup \Omega(G)$. For every edge $e=uv\in E$, we add a set $\mathring{e}=(u,v)$ of continuum many points, making these sets $\mathring{e}$ disjoint from each other and from $V\cup \Omega$. We then choose for each $e$ some fixed bijection between $\mathring{e}$ and the real interval $(0,1)$, and extend this bijection to one between $[u,v]:=\{u\}\cup\mathring{e}\cup\{v\}$ and $[0,1]$. We call $[u,v]$ a topological edge with inner points $x\in \mathring{e}$. Given any $F\subseteq E$ we write $\mathring{F}:=\cup\{\mathring{e}:e\in F\}$. When we speak of a ``graph" $H\subseteq G$, we shall often also mean its corresponding point set $V(H)\cup \mathring{E}(H)$.

Having defined the point set of $|G|$, we now choose a basis of open sets to define its topology. For every edge $uv$, declare as open all subsets of $(u,v)$ that correspond, by our fixed bijection between $(u,v)$ and $(0,1)$, to an open set in $(0,1)$. For every vertex $u$ and $\epsilon>0$, declare as open the ``open star around $u$ of radius $\epsilon$", that is, the set of all points on edges $[u,v]$ at distance less than $\epsilon$ from $u$, measured individually for each edge in its metric inherited from $[0,1]$. Finally, for every end $\omega$ and every finite set $S\subseteq V$, there is a unique component $\omega(S)$ of $G\setminus S$. Let
\begin{align*}
\Omega(S,\omega):=\{\omega'\in \Omega: \omega'(S)=\omega(S)\}.
\end{align*}
For every $\epsilon>0$, write $\mathring{E}_{\epsilon}(S,\omega)$ for the set of all inner points of edges joining a vertex in $S$ and a vertex in $\omega(S)$ at distance less than $\epsilon$ from their endpoint in $\omega(S)$. Then declare as open all sets of the form
\begin{align}
C_{\epsilon}(S,\omega):=\omega(S)\cup \Omega(S,\omega)\cup \mathring{E}_{\epsilon}(S,\omega).\label{dfce}
\end{align}
This completes the definition of $|G|$, whose open sets are the unions of the sets we explicitly chose as open above.
$|G|$ is called the Freudenthal compactification of $G$.

\begin{proposition}(Proposition 8.6.1 in \cite{Diestel2017})If $G$ is connected and locally finite, then $|G|$ is a compact Hausdorff space.
\end{proposition}


\begin{lemma}\label{l23}Let $G$ be a connected, locally finite, planar graph, then there exists a homeomorphism from $|G|$, the Freudenthal compactification of $G$, onto a subset of $\mathbb{S}^2$.
\end{lemma}

\begin{proof}See Proposition 1.22 of \cite{JPM23}; see also Page 45 of \cite{hb09}, after Theorem 4.4.
\end{proof}


\begin{definition}\label{df31}
Let $G=(V,E)$ be an infinite, connected, locally finite, planar graph. Given an embedding $\phi:G\rightarrow \mathbb{S}^2$, $a\in \mathbb{S}^2$ is an accumulation point if its arbitrary neighborhood intersects infinitely many edges of $G$. Let $\mathcal{A}$ be the set of all accumulation points. The embedding $\phi$ is called well-separated if $\phi(G)$ is disjoint from $\mathcal{A}$.
\end{definition}

\begin{proposition}[Freudenthal embedding is well--separated and matches ends]\label{th52}
Let $G=(V,E)$ be an infinite, connected, locally finite planar graph. Then there exists a planar embedding
$\phi:G\to\mathbb{S}^2$ with the following properties:
\begin{enumerate}[label=(\roman*)]
\item \textbf{(Well--separated)} Each edge $e\in E$ is mapped homeomorphically onto a simple arc,
distinct edges meet only at common endpoints, and
\[
 \mathcal{A}\cap \phi(G)=\emptyset;
\]
where $\mathcal{A}$ is the set of accumulation points of $G$ under the embedding $\phi$.
\item \textbf{(Ends $\leftrightarrow$ accumulation points)} There exists a topological embedding (homeomorphism onto the image)
$\widehat\phi:|G|\hookrightarrow \mathbb{S}^2$ such that $\phi=\widehat\phi\!\upharpoonright_{|G|\setminus \Omega(G)}$ and
\[
\overline{\phi(G)}=\widehat\phi(|G|),\qquad
\mathcal{A}=\widehat\phi(\Omega(G)).
\]
In particular, the map $\Omega(G)\to \mathcal{A}$ given by $\omega\mapsto \widehat\phi(\omega)$ is a bijection.
\end{enumerate}
\end{proposition}

\begin{proof}
\textbf{Step 1: Existence of a planar embedding of $G$.}
Since $G$ is locally finite and planar, the Freudenthal compactification $|G|$ admits a topological
embedding into the sphere; that is, there exists a topological embedding
\[
\widehat\phi:\,|G|\hookrightarrow \mathbb{S}^2
\]
whose restriction to $G$ is a planar embedding of $G$; see Lemma \ref{l23}.


\smallskip
\textbf{Step 2: $\phi$ is well--separated.}
Because $\widehat\phi$ is a topological embedding of the $1$--complex $|G|$, each edge of $G$
(the image of a closed interval) is mapped to a simple arc, and distinct edges meet only at common
endpoints. Moreover, $G$ is open in $|G|$ and locally a finite union of arcs (by local finiteness
of $G$), hence every point of $G$ has a neighbourhood in $|G|$ that is disjoint from the boundary
$\Omega(G)$. Applying $\widehat\phi$, every point of $\phi(G)=\widehat\phi(|G|)$ has a neighbourhood
in $\mathbb{S}^2$ disjoint from $\widehat\phi(\Omega(G))$. Therefore no point of $\phi(G)$ lies in
$\mathcal{A}$, i.e.\ $\mathcal{A}\cap\phi(G)=\varnothing$. This is exactly
the well--separatedness in Definition \ref{df31}.

\smallskip
\textbf{Step 3: Ends $\leftrightarrow$ accumulation points.}
Since $G$ is dense in $|G|$ and $\widehat\phi$ is continuous, we have
\[
\overline{\phi(G)}=\widehat\phi(|G|).
\]
Thus
\[
\mathcal{A}=\overline{\phi(G)}\setminus \phi(G)
=\widehat\phi(|G|)\setminus \widehat\phi(G)
=\widehat\phi\!\big(|G|\setminus G\big)
=\widehat\phi(\Omega(G)).
\]
Because $\widehat\phi$ is injective, $\Omega(G)\to \mathcal{A}$ is bijective. This proves (ii).
\end{proof}

The Freudenthal embedding plays an essential role in this paper. In particular,
the following two lemmas show that the existence of an infinite percolation
cluster yields an end with a nested family of infinite subclusters, and hence,
under the Freudenthal embedding, an accumulation point $a\in\mathbb{S}^2$ at
which every neighbourhood contains an infinite subcluster.

\begin{lemma}\label{lem:end-ray}
Let $G=(V,E)$ be an infinite, connected, locally finite graph, and let $\xi$
be an infinite connected subgraph of $G$. Then there exists an end
$e\in\Omega(G)$ and, for every finite vertex set $K\subset V$, an infinite
connected subgraph
\[
  U(K) \subset e(K)\cap\xi
\]
such that
\[
  K_0 \subseteq K \quad\Longrightarrow\quad U(K) \subseteq U(K_0).
\]
\end{lemma}

\begin{proof}
Since $\xi$ is infinite and connected and $G$ is locally finite, $\xi$
contains an infinite self-avoiding path $(z_n)_{n\ge1}$.

For each finite $K\subset V$ there exists $N(K)$ such that
$\{z_n:n\ge N(K)\}\cap K=\varnothing$. Define $e(K)$ to be the (necessarily
infinite) connected component of $G\setminus K$ that contains the tail
$\{z_n:n\ge N(K)\}$. If $K_0\subseteq K$, then $G\setminus K\subseteq G\setminus K_0$,
so $e(K)$ is a connected subgraph of $G\setminus K_0$; moreover, for all large
$n$ the tail of $(z_n)$ lies in both $e(K)$ and $e(K_0)$. Since $e(K_0)$ is the
unique connected component of $G\setminus K_0$ containing this tail, we must
have $e(K)\subseteq e(K_0)$. Thus $e$ satisfies Definition~\ref{df21} and is
an end of $G$.

For each finite $K\subset V$, the set $\{z_n:n\ge N(K)\}$ is an infinite
connected subset of $e(K)\cap\xi$, so $e(K)\cap\xi$ has an infinite connected
component; denote by $U(K)$ the infinite connected component of $e(K)\cap\xi$
containing this tail. If $K_0\subseteq K$, then $e(K)\subseteq e(K_0)$, hence
$e(K)\cap\xi\subseteq e(K_0)\cap\xi$ and in particular
$\{z_n:n\ge N(K)\}\subset U(K_0)$. Since $U(K)$ is the infinite connected
component of $e(K)\cap\xi$ containing the same tail, we have
$U(K)\subseteq U(K_0)$. This is the desired nesting.
\end{proof}

\begin{lemma}\label{lm27}
Let $G=(V,E)$ be an infinite, connected, locally finite planar graph, and let
$\phi:G\to\mathbb{S}^2$ be a Freudenthal embedding as in
Proposition~\ref{th52}. Let $\xi$ be an infinite connected subgraph of $G$.
Then there exists an accumulation point $a\in\mathbb{S}^2$ of $\phi(G)$ such
that for every open neighbourhood $U\subset\mathbb{S}^2$ of $a$, the set
$U\cap\phi(\xi)$ contains an infinite connected subgraph.
\end{lemma}

\begin{proof}
Apply Lemma~\ref{lem:end-ray} to obtain an end $e\in\Omega(G)$ and, for every
finite $K\subset V(G)$, an infinite connected subgraph
\[
  U(K) \subset e(K)\cap\xi
\]
with the nesting property $K_0\subseteq K \Rightarrow U(K)\subseteq U(K_0)$.

By Proposition~\ref{th52}\,(ii), there exists a topological embedding
$\widehat\phi:|G|\hookrightarrow\mathbb{S}^2$ extending $\phi$ such that
\[
  \overline{\phi(G)}=\widehat\phi(|G|)\qquad\text{and}\qquad
  \mathcal{A}=\widehat\phi(\Omega(G)),
\]
where $\mathcal{A}$ is the set of accumulation points of $\phi(G)$ in the sense of
Definition~\ref{df31}. In particular, the map
\[
  \Omega(G)\to \mathcal{A},\qquad \omega\mapsto \widehat\phi(\omega),
\]
is a bijection. Let $a:=\widehat\phi(e)\in \mathcal{A}$. Then $a$ is an accumulation
point of $\phi(G)$.

Since $\widehat\phi$ is a homeomorphism from $|G|$ onto its image
$\widehat\phi(|G|)\subset\mathbb{S}^2$, the images of the basic Freudenthal
neighbourhoods at $e$ form a neighbourhood basis of $a$ in $\mathbb{S}^2$.
By the definition of the Freudenthal topology (see the description of the
sets $C_\varepsilon(S,\omega)$ above), this means that for every open
neighbourhood $U\subset\mathbb{S}^2$ of $a$ there exist a finite set
$S\subset V(G)$ and $\varepsilon>0$ such that
\[
  \widehat\phi\bigl(C_\varepsilon(S,e)\bigr) \subset U,
\]
and in particular
\[
  \phi\bigl(e(S)\bigr) \subset U,
\]
since $e(S)$ is a subset of $C_\varepsilon(S,e)$.

Fix such an $S$ and write $K:=S$. By Lemma~\ref{lem:end-ray}, $U(K)$ is an
infinite connected subgraph contained in $e(K)\cap\xi$. Applying $\phi$ and
using that $\phi$ is a homeomorphism onto its image on $G$, we obtain that
$\phi\bigl(U(K)\bigr)$ is an infinite connected subgraph of $\phi(G)$
contained in
\[
  \phi\bigl(e(K)\cap\xi\bigr) \subset \phi\bigl(e(K)\bigr)\cap\phi(\xi)
  \subset U\cap\phi(\xi).
\]
Thus every open neighbourhood $U$ of $a$ in $\mathbb{S}^2$ meets $\phi(\xi)$
in an infinite connected subgraph, as claimed.
\end{proof}

\section{End--equivalence classes and $F$--connectivity thresholds}
\label{sec3}

Fix a well--separated planar embedding of $G$ into $\mathbb S^2$.
We introduce a cycle--separation equivalence on the accumulation set: two accumulation points (ends)
are equivalent if no cycle of $G$ separates them on the sphere.
This produces the family $\mathcal F(G)$ of end--equivalence classes, each of which is compact.
For every $F\in\mathcal F(G)$ we then define the corresponding $F$--connectivity threshold
$p^{\mathrm{site}}_{c,F}(G)$ governing percolative connections aimed at $F$.
The remainder of the paper splits according to whether $\mathcal F(G)$ is countable or uncountable,
since these regimes exhibit different coexistence behavior.

\begin{definition}[Cycle]
Let $G=(V,E)$ be a planar graph. A \emph{cycle} in $G$ is a finite sequence of distinct
vertices $(v_0,\ldots,v_{k-1})$ with $k\ge 3$ such that
$\{v_i,v_{i+1}\}\in E$ for $i=0,\ldots,k-2$ and $\{v_{k-1},v_0\}\in E$.
Equivalently, a cycle is a finite connected $2$-regular subgraph of $G$.

If $G$ is equipped with a planar embedding $\phi\colon G\hookrightarrow \mathbb S^2$,
we also identify a cycle with its geometric image
$\phi(v_0v_1\cdots v_{k-1}v_0)$, which is a simple closed curve in $\mathbb S^2$
consisting of finitely many edges of $G$.
\end{definition}

\begin{definition}[Equivalence classes of accumulation points]\label{df52}
Let $G=(V,E)$ be a connected, locally finite infinite graph and let
$\phi:G\hookrightarrow\mathbb S^2$ be a well--separated embedding.  
Let $A(\phi)$ denote the set of accumulation points of $\phi(G)$
(as in Definition~\ref{df31}).

We say that a cycle $C$ of $G$ \emph{separates} two points
$x,y\in \mathcal{A}(\phi)$ if $C\cap A(\phi)=\varnothing$ and $x,y$ lie in distinct
connected components of $\mathbb S^2\!\setminus C$.
Define a relation $\sim_\phi$ on $A(\phi)$ by
\[
x\sim_\phi y
\quad\Longleftrightarrow\quad
\text{no cycle of $G$ separates $x$ and $y$ in the above sense.}
\]
Equivalently, for every cycle $C$ with $C\cap A(\phi)=\varnothing$,
the points $x$ and $y$ belong to the same connected component of
$\mathbb S^2\!\setminus C$.
\end{definition}

\begin{lemma}\label{lem:equiv}
The relation $\sim_\phi$ on $A(\phi)$ is an equivalence relation.
\end{lemma}

\begin{proof}
Reflexivity and symmetry are immediate from the definition.

For transitivity, let $x,y,z\in A(\phi)$ with $x\sim_\phi y$ and $y\sim_\phi z$.
Suppose, towards a contradiction, that $x\not\sim_\phi z$. Then there exists a
cycle $C$ with $C\cap A(\phi)=\varnothing$ such that $x$ and $z$ lie in distinct
components of $\mathbb S^2\!\setminus C$. Since $C$ is disjoint from $A(\phi)$,
in particular $y\notin C$; hence $y$ lies in one of the two components of
$\mathbb S^2\!\setminus C$. If $y$ is in the component of $x$, then $C$ separates
$y$ from $z$, contradicting $y\sim_\phi z$. If $y$ is in the component of $z$,
then $C$ separates $x$ from $y$, contradicting $x\sim_\phi y$. Thus no such $C$
exists and $x\sim_\phi z$.
\end{proof}

\begin{lemma}\label{le64}
Let $G$ be connected and locally finite, and $\phi:G\hookrightarrow\mathbb S^2$ a well–separated embedding. 
Let $\sim_\phi$ be the cycle–separation equivalence on the accumulation set $\mathcal{A}(\phi)$ as in Definition~\ref{df52}.
For every equivalence class $[x]\subset \mathcal{A}(\phi)$, the subset $[x]$ is closed in $\mathbb S^2$. 
Consequently $[x]$ is compact.
\end{lemma}

\begin{proof}
Let $(y_n)\subset [x]$ with $y_n\to y\in \mathcal{A}(\phi)$. Suppose $y\notin [x]$. 
Then, by definition of $\sim_\phi$, there exists a cycle $C$ with $C\cap \mathcal{A}(\phi)=\varnothing$ such that $x$ and $y$ lie in different components of $\mathbb S^2\setminus C$. The two components are open; hence for all large $n$, $y_n$ lies in the same component as $y$, and thus $C$ separates $x$ from $y_n$. This contradicts $y_n\in[x]$. Hence $y\in[x]$, so $[x]$ is closed in $\mathcal{A}(\phi)$, in particular in $\mathbb S^2$. As $\mathbb S^2$ is compact, $[x]$ is compact.
\end{proof}

Let $G$ be an infinite, connected, locally finite planar graph. We identify $G$ with its Freudenthal embedding into $\mathbb{S}^2$.

\begin{definition}Let
\begin{align*}
\mathcal{F}:=\mathcal{A}/\sim_{\phi}
\end{align*}
be the set of equivalent end classes.

\end{definition}

\begin{definition}\label{df36}For each $F\in \mathcal{F}$, define
\begin{align*}
p_{c,F}^{site}(G):=\inf\{p:\PP_p(\cup_{v\in V}\cup_{a\in F}\{v\leftrightarrow a\})=1\}.
\end{align*}
Here $v\leftrightarrow a$ means there exists an infinite self-avoiding path
\begin{align*}
z_0(=v),z_1,z_2,\ldots,z_n,\ldots
\end{align*}
such that 
\begin{align*}
\lim_{n\rightarrow\infty}z_n=a.
\end{align*}
This is possible by Lemma \ref{lm27}.
\end{definition}

\begin{definition}We call $a\in \mathcal{A}$ active if $v\leftrightarrow a$ occurs or some $v\in V$. We call $F\in \mathcal{F}$ active if there exists $a\in F$ which is active.
\end{definition}

The following cases might occur
\begin{enumerate}
\item $1-p>\inf_{F\in \mathcal{F}}p_{c,F}^{site}(G)$;
\item 
\begin{align}
\pcs(G)<1-p\leq \inf_{F\in \mathcal{F}}p_{c,F}^{site}(G)\label{cs}
\end{align}
\end{enumerate}

\begin{lemma} (\ref{cs}) occurs only when $\mathcal{F}$ is uncountable. Moreover, when (\ref{cs}) holds, for any countable subset $\mathcal{F}_0\subset \mathcal{F}$;
\begin{align}
\PP_p(\cup_{F\in \mathcal{F}\setminus \mathcal{F}_0}\cup_{v\in V}\cup_{a\in F}\{v\xleftrightarrow{0} a\})=1.\label{cd2}
\end{align}
\end{lemma}

\begin{proof}Note that (\ref{cs}) implies that $\pcs(G)< \inf_{F\in \mathcal{F}}p_{c,F}^{site}(G)$. Then there exists $p\in (0,1)$ with $\pcs(G)<1-p< \inf_{F\in \mathcal{F}}p_{c,F}^{site}(G)$. For each $F\in \mathcal{F}$,
\begin{align}
\PP_p(\cup_{v\in V}\cup_{a\in F}\{v\xleftrightarrow{0} a\})=0.\label{cd1}
\end{align}
If $\mathcal{F}$ is countable, then (\ref{cd1}) and Lemma \ref{lm27} imply a.s.~there are no infinite 0-clusters. This contradicts the fact that $1-p>\pcs(G)$; then $F$ must be uncountable and (\ref{cd2}) holds when $\pcs(G)<1-p< \inf_{F\in \mathcal{F}}p_{c,F}^{site}(G)$ follows.

When $1-p=\inf_{F\in \mathcal{F}}p_{c,F}^{site}(G)$, (\ref{cd2}) follows from stochastic domination.
\end{proof}

Let $\{v\xleftrightarrow{1} F\}=\cup_{a\in F}\{v\xleftrightarrow{1} a\}$.
Then the following cases might occur
\begin{enumerate}[label=(\alph*)]
\item Let $\mathcal{A}_1\subset \mathcal{A}$ consists of all the ends $a$ satisfying
\begin{align}
\PP_p(\cup_{v\in V}\{v\xleftrightarrow{1} a\})=1.\label{ccds}
\end{align}
then
\begin{align}
\PP_p(\cup_{a\in \mathcal{A}\setminus \mathcal{A}_1 }\cup_{v\in V}\{v\xleftrightarrow{1} a\})=1.\label{sst}
\end{align}

\item 
\begin{align}
\PP_p(\cup_{a\in \mathcal{A}\setminus \mathcal{A}_1 }\cup_{v\in V}\{v\xleftrightarrow{1} a\})=0.\label{sst1}
\end{align}
\end{enumerate}

\begin{lemma}\label{le76}Let $G=(V,E)$ be a planar, connected, locally finite graph, identified with its Freudenthal embedding into $\mathbb{S}^2$. Consider a tail trivial
site percolation process $\sigma$ on G. Let $\mathcal{N}_{\infty}$ be the number of infinite 1-clusters.
If
\begin{align}
\PP_p(1\leq \mathcal{N}_{\infty}<\infty)=1.\label{fi1}
\end{align} 
Then, there exists an accumulation point (or equivalently an end), $a\in \mathcal{A}$, such that a.s.~$\cup_{v\in V}\{v\leftrightarrow a\}$ occurs. Moreover, if $C_v$ is the 1-cluster at $v$ and $a\in \mathcal{A}$ such that $a\in \overline{C_v}$, then $\{v\leftrightarrow a\}$ occurs.
\end{lemma}

\begin{proof}For $v\in V$, let $C_v$ be the 1-cluster containing $v$. By Lemma 2.3 in \cite{GHZ25}, there exists $a\in A$, such that a.s.~$\cup_{v\in V}\{a\in \overline{C}_v\}$, where $C_v$ is the closure of the set $C_v\subset \mathbb{S}$. Under the Freudenthal embedding of $G$, each accumulation point $a\in A$ is identified with an end of $G$. More precisely, for any finite subset $K\subset V$, there is a unique infinite component of $G\setminus K$ corresponding to $a$, denoted by $a(K)$. By the definition of open sets (\ref{dfce}) and the fact that the Freudenthal embedding is a homeomorphism from $|G|$ to a subset of $\mathbb{S}^2$ (Lemma \ref{l23}), $a(K)\cap C_v$ contains infinitely many vertices since $a\in \overline{C}_v$.  It follows that $a(K)\cap C_v$ must contain an infinite component since $K$ is finite. Then the lemma follows since $K$ is arbitrary.
\end{proof}

\section{Boundary decomposition and alternating arm events}
\label{sec4}

This section develops the geometric primitives for the multi--arm exploration.
Given a finite connected vertex set $S$ and an end--equivalence class $F$, we consider the component
of $\mathbb S^2\setminus G_S$ containing $F$ and define an $F$--adapted boundary cycle
$\partial_F S$ by a wedge--neighbor enumeration followed by a pruning procedure that removes
redundant repetitions toward $F$.
The resulting cyclic sequence provides a canonical set of boundary arcs along which we define
alternating arm events in $S^c$ aimed at $F$.  These events encode the presence of many disjoint
corridors separating different regions of the complement and will be the topological input forcing
multiple infinite clusters.

\begin{definition}\label{df26}
Let \(G=(V,E)\) be an infinite, connected, locally finite planar graph, and let
\(S\subset V\) be finite with the induced subgraph \(G_S\) connected. Fix a well–separated Freudenthal
 embedding \(\phi:G\hookrightarrow\mathbb S^2\), and
identify \(G\) with \(\phi(G)\subset\mathbb S^2\). We also identify ends of $G$ with the accumulation points of the embedding $\phi$. Let $F\in \mathcal{F}$ be an end equivalent class.

Let $Q$ be the connected component of \(\mathbb S^2\!\setminus\!G_S\) 
containing $F$. Define
\begin{align}\label{dds}
S_Q\ :&=\ \Bigl\{\,u\in (V\setminus S)\cap Q:\ \exists\,y\in S\ \text{with }(u,y)\in E
\ \text{and}\ \exists\,x\in(Q\cap V)^\circ\ \text{with }(u,x)\in E\\
&\text{ and a singly infinite path in }(Q\cap V)^\circ\text{ from }x\,\Bigr\}.\notag
\end{align}
Here for each $U\subset V$, $U^{\circ}$ consists of all the vertices in $U$ with all the neighbors in $U$ as well.

Note that \(Q\) is a topological \(2\)–disc whose frontier \(\partial Q\) is the image of a closed walk in \(G_S\) (vertices/edges may repeat).
List the vertices on \(\partial Q\cap V\) as
\begin{equation}\label{dqv}
v_1,\ldots,v_n
\end{equation}
in clockwise order around an interior point of \(Q\) (a vertex may appear more than once
along \(\partial Q\)); with \((v_j,v_{j+1})\in E\) for \(1\le j\le n-1\) and \((v_n,v_1)\in E\),
and \(n\) minimal among such listings.

For each \(j\in\{1,\dots,n\}\), define the \emph{wedge neighbors} of \(v_j\) in \(Q\) by
\[
\mathcal N_j(Q):=\{\,w_{j,1},\ldots,w_{j,k_{j}}\,\}\subset S_Q,
\]
where \((v_j,w_{j,1}),\ldots,(v_j,w_{j,k_{j}})\) is the counterclockwise order (around \(v_j\))
of all edges from \(v_j\) into \(Q\) whose other endpoint lies strictly \emph{between} the boundary edges
\((v_j,v_{j-1})\) and \((v_j,v_{j+1})\) (indices modulo \(n\)) and in $S_Q$.
Concatenate these finite lists in the boundary order
\begin{equation}\label{pni}
w_{1,1},\ldots,w_{1,k_{1}},
w_{2,1},\ldots,w_{2,k_{2}},\ \ldots,\
w_{n,1},\ldots,w_{n,k_{n}}.
\end{equation}
If for some \(j\) the last element of \(\mathcal N_j(Q)\) equals the first element of \(\mathcal N_{j+1}(Q)\)
(with \(n{+}1:=1\)), delete one of these two equal consecutive terms. Denote the resulting cyclic
sequence by \(\mathcal N(Q)\).

Let \(m:=|\mathcal N(Q)|\) and write \(\mathcal N(Q)=(u_1,\dots,u_m)\).

\smallskip
We now prune repetitions toward $F$.
Set $Q_0:=Q$ and $\mathcal N_0:=\mathcal N(Q)$.
Given $(Q_r,\mathcal N_r)$, if $\mathcal N_r$ has pairwise distinct entries we stop.
Otherwise, let $u$ be the first vertex (in the linear representative $u_1,\dots,u_m$)
that appears at least twice, and let $u_a=u_b=u$ be two \emph{consecutive occurrences}
of $u$ in the cyclic order (i.e.\ $u$ does not occur among $u_{a+1},\dots,u_{b-1}$).
Let $v_x$ (resp.\ $v_y$) be the boundary vertex in $\partial Q$ whose wedge list
contains the occurrence $u_a$ (resp.\ $u_b$).
Let $Q_{r+1}$ be the connected component of
\[
Q_r\setminus\bigl(\phi(v_xu)\cup \phi(v_yu)\bigr)
\]
containing $F$, and define $\mathcal N_{r+1}:=\mathcal N(Q_{r+1})$
by repeating the above wedge construction with $Q$ replaced by $Q_{r+1}$.

Since $|\mathcal N_{r+1}|<|\mathcal N_r|$ at each pruning step, the process terminates.
Let $k$ be the first index for which $\mathcal N_k$ has pairwise distinct entries.
We define the \emph{$F$--boundary of $S$} by
\[
\partial_F S := \mathcal N_k,
\]
viewed as a cyclic sequence.
\end{definition}

\begin{definition}\label{df443}
Assume
\[
\partial_F S=(u_1,\dots,u_h),
\]
in cyclic order as in Definition \ref{df26}.
Given indices \(1\le i_1<\cdots<i_{2k}\le h\), define consecutive boundary subarcs
\[
\mathrm{Arc}_j:=\{u_{i_{j-1}+1},\dots,u_{i_j}\}\subset\partial_{F}S,\qquad i_0:=i_{2k}.
\]
For a site configuration \(\sigma\in\{0,1\}^V\) at parameter \(p\), and $a\in Q\cap A$, the \emph{alternating–arm event}
across \(S\) with respect to $a$ is
\begin{equation}\label{eq:def-Arm}
\mathrm{Arm}(i_1,\dots,i_{2k};S;a)
:= \bigcap_{j=1}^{2k}\Big(
  \{\mathrm{Arc}_j \xleftrightarrow{S^{\mathrm c}\cap\{\sigma=1\}} F\}
  \ \cap\
  \{\mathrm{Arc}_j \xleftrightarrow{S^{\mathrm c}\cap\{\sigma=0\}} F\}
\Big).
\end{equation}
\end{definition}

\section{Cutset estimates and $F$--connectivity via the $\phi$--functional}
\label{sec5}

The goal of this section is to supply quantitative control of connectivity to a prescribed end class.
We develop a cutset--based differential inequality for $F$--connectivity and recast it in terms of the
$\phi$--functional, yielding an analytic characterization of $p^{\mathrm{site}}_{c,F}(G)$ that is robust
under truncations excluding neighborhoods of other ends.
We also derive uniform supercritical bounds that allow one to pass from an infinite cluster aimed at $F$
at one parameter to high--probability $F$--connections from many well--separated vertices at a larger
parameter.  In the special case where $\mathcal F(G)$ is countable, these estimates imply the identity
$p^{\mathrm{site}}_c(G)=\inf_{F\in\mathcal F(G)}p^{\mathrm{site}}_{c,F}(G)$, which ties the global critical point
to the family of end--directed thresholds.

\begin{assumption}\label{ap36} Assume $\mathcal{F}$ is countable.
\end{assumption}

\begin{lemma}Let $G=(V,E)$ be an infinite, connected, locally finite planar graph. Under Assumption \ref{ap36}, we have
\begin{align*}
\pcs(G)=\inf_{F\in\mathcal{F}}p_{c,F}^{site}(G).
\end{align*}
\end{lemma}

\begin{proof}From the definitions of $\pcs(G)$ and $p_{c,F}^{site}(G)$, it is straightforward to see that
\begin{align*}
\pcs(G)\leq\inf_{F\in\mathcal{F}} p_{c,F}^{site}(G).
\end{align*}
It suffices to show that whenever $p>\pcs(G)$, there exists $F\in\mathcal{F}$ such that $p\geq p_{c,F}^{site}(G)$. 

Indeed 
when $p>\pcs(G)$, by Kolmolgorov 0-1 law we have 
\begin{align*}
\PP_p(\cup_{v\in V}\cup_{F\in \mathcal{F}}\cup_{a\in F}\{v\leftrightarrow a\})=1
\end{align*}
By Assumption \ref{ap36}, $\mathcal{F}$ is countable, it follows that there exists $F\in \mathcal{F}$, s.t.
\begin{align*}
\PP_p(\cup_{v\in V}\cup_{a\in F}\{v\leftrightarrow a\})>0
\end{align*}
Since $\cup_{v\in V}\cup_{a\in F}\{v\leftrightarrow a\}$ is tail-measurable, we have
\begin{align*}
\PP_p(\cup_{v\in V}\cup_{a\in  F}\{v\leftrightarrow a\})=1,
\end{align*}
then $p\geq p_{c,F}^{site}(G)$, and the lemma follows.
\end{proof}

The following lemma is straightforward.

\begin{lemma}\label{le63}Let $G=(V,E)$ be an infinite, connected, locally finite planar graph, and let
$\phi:G\to\mathbb{S}^2$ be a well-separated embedding. If $p>\inf_{F\in\mathcal{F}}p_{c,F}^{site}(G)$, then there exists $F\in \mathcal{F}$, s.t. $\PP_p$-a.s.~$F$ is active.
\end{lemma}

Let $F\in\mathcal{F}$ and $p>p_{c,F}^{site}(G)$. Then ~$\PP_p$-a.s.~$F$ is active. For each $v\in V$, define 
\begin{align*}
&D_F^{\epsilon}:=\cup_{a\in F}B_{\mathbb{S}^2}(a,\epsilon);\\ 
&A_{F,c}^{\epsilon}:=\mathcal{A}\setminus D_F^{\epsilon}.\\
&D_{F,c}^{\epsilon}:=\mathbb{S}^2\setminus D_F^{\epsilon};\\
&\partial^V D_F^{\epsilon}:=\{u\in V\cap D_{F}^{\epsilon}:\exists w\notin V\cap D_{F}^{\epsilon},~s.t.~uw\in E\}.
\end{align*}
Here we write $d_{\mathbb{S}^2}$ for the geodesic (great–circle) distance on $\mathbb{S}^2$ induced by
the round metric. For $a\in\mathbb{S}^2$ and $\varepsilon>0$,
\[
B_{\mathbb{S}^2}(a,\varepsilon)
:=\{x\in\mathbb{S}^2:\ d_{\mathbb{S}^2}(x,a)\le \varepsilon\}.
\]

For any positive integers $s,t\geq 1$, define
\begin{align*}
V_{s,t}=\left\{V\setminus \left[D_{F}^{\frac{1}{s}}\cup\cup_{a\in A_{F,c}^{\frac{1}{s}}}B_{\mathbb{S}^2}\left(a,\frac{1}{t}\right)\right]\right\}\cup \partial^V D_F^{\frac{1}{s}}
\end{align*}
Here for $T\subseteq \mathbb{S}^2$, $G\setminus T$ is the subgraph of $G$ induced by all the vertices not in $T$.
Since $a$ is an accumulation point of a well-separated embedding $\phi$ of $G$, it follows that $a\notin \phi(G)$ for any $a\in A$. 
Then we have
\begin{lemma}\label{l39}Suppose the assumptions of Lemma \ref{le63} hold.
Assume $F\in \mathcal{F}$.
\begin{align*}
\cup_{a\in  F}\{v\leftrightarrow a\}=\cap_{s=1}^{\infty}\cup_{t=1}^{\infty}\{v\xleftrightarrow{V_{s,t}\cap\{\sigma=1\}} \partial^V D_{F}^{\frac{1}{s}}\}.
\end{align*}
\end{lemma}
\begin{proof}Note that for any $s\geq 1$, any infinite path starting from $v$ and converging to an accumulation point in $F$ has finitely many vertices $\{u_1,\ldots,u_l\}$ in $D_{F,c}^{\frac{1}{s}}$. Since the graph has a well-separated embedding into $\mathbb{S}^2$. For each $1\leq i\leq l$, there exists $\epsilon_i>0$, such that
\begin{align*}
B(u_i,\epsilon_i)\cap \mathcal{A}=\emptyset.
\end{align*}
Let $t:=\max_{1\leq i\leq l}\frac{1}{\epsilon_i}$, we obtain that 
\begin{align*}
\left[\cup_{a\in  F}\{v\leftrightarrow a\}\right]\subseteq\left[\cap_{s=1}^{\infty}\cup_{t=1}^{\infty}\{v\xleftrightarrow{V_{s,t}\cap\{\sigma=1\}} \partial^V D_{F}^{\frac{1}{s}}\}\right].
\end{align*}
Now we prove 
\begin{align*}
\left[\cup_{a\in  F}\{v\leftrightarrow a\}\right]\supseteq\left[\cap_{s=1}^{\infty}\cup_{t=1}^{\infty}\{v\xleftrightarrow{V_{s,t}^v\cap\{\sigma=1\}} \partial^V D_{F}^{\frac{1}{s}}\}\right].
\end{align*}
It suffices to show that 
\begin{align}
\left[\cup_{a\in  F}\{v\leftrightarrow a\}\right]^c\subseteq\left[\cap_{s=1}^{\infty}\cup_{t=1}^{\infty}\{v\xleftrightarrow{V_{s,t}\cap\{\sigma=1\}} \partial^V D_{F}^{\frac{1}{s}}\}\right]^c.\label{rd}
\end{align}
Indeed for any configuration in $\left[\cup_{a\in  F}\{v\leftrightarrow a\}\right]^c$, one of the following two cases occurs
\begin{enumerate}
\item The 1-cluster $C$ at $v$ is finite; or
\item The 1-cluster $C$ at $v$ is infinite, but $\overline{C}\cap F=\emptyset$.
\end{enumerate}
In either case we have $\mathbb{S}^2\setminus \overline{C}$ is open, and $F\subset \mathbb{S}^2\setminus  \overline{C}$. By Lemma \ref{le64}, $F$ is compact. Since $\mathbb{S}^2$ is a metric space, we have
\begin{align*}
d_0:=d_{\mathbb{S}^2}(F,C)>0.
\end{align*}
Choose $s$ such that $\frac{1}{s}<\frac{d_0}{2}$, we obtain
\begin{align*}
v\nleftrightarrow  \partial^V D_F^{\frac{1}{s}}
\end{align*}
Then (\ref{rd}) follows.
\end{proof}

\begin{corollary}\label{c310}Suppose that the Assumptions of Lemma \ref{l37} hold. For any $v\in V$ and $F\in \mathcal{F}$, we have
\begin{align}
\label{ti}\left(v\xleftrightarrow{V_{s,t}\cap\{\sigma=1\}} \partial^V D_{F}^{\frac{1}{s}}\right)\subseteq \left(v\xleftrightarrow{V_{s,t+1}\cap\{\sigma=1\}} \partial^V D_{F}^{\frac{1}{s}}\right)
\end{align}
and 
\begin{align}
\label{sd}\left(v\xleftrightarrow{V_{s+1,t}\cap\{\sigma=1\}} \partial^V D_{F}^{\frac{1}{s+1}}\right)\subseteq \left(v\xleftrightarrow{V_{s,t}\cap\{\sigma=1\}} \partial^V D_{F}^{\frac{1}{s}}\right)
\end{align}
Moreover,
\begin{align}
\PP_p(\cup_{a\in  F}\{v\leftrightarrow a\})&=\lim_{s\rightarrow\infty}\lim_{t\rightarrow\infty}\PP_p(v\xleftrightarrow{V_{s,t}\cap\{\sigma=1\}} \partial^V D_{F}^{\frac{1}{s}})\label{llst}=\inf_{s}\sup_t \PP_p(v\xleftrightarrow{V_{s,t}\cap\{\sigma=1\}} \partial^V D_{F}^{\frac{1}{s}})
\end{align}
\end{corollary}

\begin{proof}The corollary follows from Lemma \ref{l39} in a straightforward way.
\end{proof}

\subsection{Definition of the $\varphi_p^v(S)$-functional}

Let $G=(V,E)$ be a graph. For each $p\in (0,1)$, let $\mathbb{P}_p$ be the probability measure of the i.i.d.~Bernoulli($p$) site percolation on $G$.
For each $S\subset V$, let $S^{\circ}$ consist of all the interior vertices of $S$, i.e., vertices all of whose neighbors are in $S$ as well.
For each  $S\subseteq V$, $v\in S$, define
\begin{align}
\varphi_p^{v}(S):=\begin{cases}\sum_{y\in S:[\partial_V y]\cap S^c\neq\emptyset}\mathbb{P}_p(v\xleftrightarrow{S^{\circ}} \partial_V y)&\mathrm{if}\ v\in S^{\circ}\\
1&\mathrm{if}\ v\in S\setminus S^{\circ}
\end{cases}\label{dpv}
\end{align}
where 
\begin{itemize}
\item $v\xleftrightarrow{S^{\circ}} x$ is the event that the vertex $v$ is joined to the vertex $x$ by an open path visiting only interior vertices in $S$;
\item let $A\subseteq V$; $v\xleftrightarrow{S^{\circ}} A$ if and only if there exists $x\in A$ such that $v\xleftrightarrow{S^{\circ}} x$;
\item $\partial_V y$ consists of all the vertices adjacent to $y$.
\end{itemize}

\begin{lemma}\label{l311}Suppose that the Assumptions of Lemma \ref{le63} hold. For $F\in\mathcal{F}$, define
\begin{align*}
    \tilde{p}_{c,F}=&\sup\{p\geq 0:\exists \epsilon_0>0, \mathrm{s.t.}\forall v\in V, \exists S_v\subseteq V\ \mathrm{satisfying}\ |S_v|<\infty\ \mathrm{and}\ v\in S_v^{\circ},\\
    &\inf_s\sup_t\varphi_p^{v}(S_v;G_{s,t}^{\infty}(F))\leq 1-\epsilon_0\}
\end{align*}
Here 
\begin{align*}
G_{s,t}^{\infty}(F)=G\setminus \left[\cup_{a\in A_{F,c}^{\frac{1}{s}}}B_{\mathbb{S}^2}\left(a,\frac{1}{t}\right)\right];
\end{align*}
and $\varphi_p^{v}(S_v;G_{s,t}^{\infty}(F))$ is defined as in (\ref{dpv}) with respect to the graph $G_{s,t}^{\infty}(F)$.
Then the following statements hold:
\begin{enumerate}
    \item If $p>\tilde{p}_{c,F}$, 
    \begin{align}
    \PP_p(\cup_{v\in V}\cup_{a\in F}\{v\leftrightarrow a\})=1.\label{pf1}
    \end{align}
    moreover, for any $\epsilon>0$ there exists a vertex $w$,  such that 
    \begin{align}
    \label{swc2}
    &\forall S_w\subseteq V\ \mathrm{satisfying}\ |S_w|<\infty\ \mathrm{and}\ w\in S_{w}^{\circ},\\
    &\inf_s\sup_t\varphi_q^{v}(S_w;G_{s,t}^{\infty}(F))> 1-\epsilon_1;\ \forall q\geq p_1\notag
    \end{align}
    where $p_1,\epsilon_1$ is such that 
    \begin{align}
    p_1\in (\tilde{p}_{c,F},p);\ \epsilon_1\in(0,\epsilon);\ \left(\frac{1-p}{1-p_1}\right)^{1-\epsilon_1}<\left(\frac{1-p}{1-\tilde{p}_{c,F}}\right)^{1-\epsilon}.\label{pepe}
    \end{align}
    Any vertex $w$ satisfying (\ref{swc2}) also satisfies
    \begin{align}
    \mathbb{P}_p(\cup_{a\in  F}\{w\leftrightarrow a\})\geq 1-\left(\frac{1-p}{1-p_1}\right)^{1-\epsilon_1}\geq 1-\left(\frac{1-p}{1-\tilde{p}_{c,F}}\right)^{1-\epsilon}.\label{lbi2}
    \end{align}
    \item If $p<\tilde{p}_c$, then for any vertex $v\in V$
    \begin{align}
        \PP_p(\cup_{v\in V}\cup_{a\in F}\{v\leftrightarrow a\})=0.\label{lbc2}
    \end{align}
\end{enumerate}
In particular, (1) and (2) implies that $p_{c,F}^{site}(G)=\tilde{p}_{c,F}$.
\end{lemma}

\begin{lemma}\label{l312}Let $G=(V,E)$ be an infinite, connected, locally finite, planar graph with a well-separated embedding into $\mathbb{S}^2$. For all $p>0$ 
\begin{align*}
&\frac{d}{dp}\mathbb{P}_p\left(v\xleftrightarrow{V_{s,t}\cap\{\sigma=1\}} \partial^V D_{F}^{\frac{1}{s}}\right)\\
&\geq \frac{1}{1-p}
\left[ \mathrm{inf}_{S:v\in S,|S|<\infty}\varphi_p^v(S;G_{s,t}^{\infty}(F))\right]
\left(1-\mathbb{P}_p\left(v\xleftrightarrow{V_{s,t}\cap\{\sigma=1\}} \partial^V D_{F}^{\frac{1}{s}}\right)\right)
\end{align*}
\end{lemma}

\begin{proof}Let $V_{s,t}^{\infty}$ be the vertex set of $G_{s,t}^{\infty}(F)$.
The lemma follows from Lemma 2.6 in \cite{ZL24} by considering the site percolation on the graph $G_{s,t}^{\infty}(F)$ and letting 
\begin{align*}
\Lambda:=V_{s,t}^{\infty}\setminus D_F^{\frac{1}{s}}.
\end{align*}
\end{proof}

\begin{lemma}(Lemma 2.7 in \cite{ZL24})\label{l73}Let $G=(V,E)$ be an infinite, connected, locally finite graph. Let $p>0$, $u\in S\subset A$ and $B\cap S=\emptyset$. Then
\begin{itemize}
\item If $u\in S^{\circ}$
\begin{align*}
    \mathbb{P}_p(u\xleftrightarrow{A} B)\leq \sum_{y\in S:\partial_V y\cap S^c\neq \emptyset}
    \mathbb{P}_p(u\xleftrightarrow{S^{\circ}}\partial_V y )\mathbb{P}_p(y\xleftrightarrow{A} B).
\end{align*}
\item If $u\in S\setminus S^{\circ}$,
\begin{align*}
    \mathbb{P}_p(u\xleftrightarrow{A} B)\leq \sum_{y\in S:\partial_V y\cap S^c\neq \emptyset}
    \mathbf{1}_{y=u}\mathbb{P}_p(y\xleftrightarrow{A} B).
\end{align*}
\end{itemize}
\end{lemma}

\bigskip
\noindent\textbf{Proof of Lemma \ref{l311}(1).}
Let $p>\tilde{p}_{c,F}$. Define
\begin{align}
f_{v,s,t}(p):=\mathbb{P}_p\left(v\xleftrightarrow{V_{s,t}\cap\{\sigma=1\}} \partial^V D_{F}^{\frac{1}{s}}\right)\label{dfv}
\end{align}
Since $p>\tilde{p}_{c,F}$, for any $\epsilon>0$, there exists $w\in V$, such that (\ref{swc2}), (\ref{pepe}) holds. Note that
\begin{align}
\varphi_q^v(S_w,G_{s,t}^{\infty}(F))\leq \varphi_q^v(S_w,G_{s,t+1}^{\infty}(F))\label{tis}
\end{align}
and
\begin{align*}
\varphi_q^v(S_w,G_{s+1,t}^{\infty}(F))\leq \varphi_q^v(S_w,G_{s,t}^{\infty}(F))
\end{align*}

Given (\ref{swc2}), for all $s\geq 1$ 
\begin{align*}
\sup_t\varphi_q^v(S_w,G_{s,t}^{\infty}(F))>1-\epsilon_1;\ \forall q\geq p_1
\end{align*}
and by (\ref{tis}) there exists $N(s)\geq 1$,  for all $t\geq N(s)$,
\begin{align*}
\varphi_q^v(S_w,G_{s,t}^{\infty}(F))>1-\epsilon_1;\ \forall q\geq p_1
\end{align*}

By Lemma \ref{l312}, for any $\epsilon>0$, there exists $w\in V$
\begin{align*}
    \frac{f_{w,s,t}'(p)}{1-f_{w,s,t}(p)}\geq \frac{1-\epsilon_1}{1-p},\ \forall p>p_1;\ \mathrm{and}\ t\geq N(s).
\end{align*}
Integrating both the left hand side and right hand side from $p_1>\tilde{p}_{c,F}$ to $p$, using the fact that $f_{w,s,t}\left(p_1\right)\geq 0$, and (\ref{pepe}) we obtain 
\begin{align*}
f_{w,s,t}(p)\geq 1-\left(\frac{1-p}{1-p_1}\right)^{1-\epsilon_1}>1-\left(\frac{1-p}{1-\tilde{p}_{c,F}}\right)^{1-\epsilon}
\end{align*}
then by (\ref{llst}) taking sup over $t$ and inf over $s$ (\ref{dfv}) follows; (\ref{pf1}) follows from the Kolmogorov 0-1 law.
\hfill$\Box$

Before proving Lemma \ref{l311}(2), we first prove a lemma
\begin{lemma}\label{l37}
Define
\begin{align}
V_{p,\epsilon}(F):=\{v\in V:\forall S_v\subset V\ \mathrm{finite\ and}\ v\in S_v^{\circ},\inf_s\sup_t\varphi_p^{v}(S_w;G_{s,t}^{\infty}(F))> 1-\epsilon\}\label{dve}
\end{align}
Let $W:=V\setminus V_{p,\epsilon}(F)$.
 Let $v\in W$ be arbitrary.
\begin{align}
\PP_{p}(v\xleftrightarrow{W}F)=0.\label{p10}
\end{align}
\end{lemma}
\begin{proof}
For any $v\in W$, there exists a finite $S_v\subseteq V$ satisfying $v\in S_v^{\circ}$ and 
\begin{align*}
\inf_{s}\sup_{t}\varphi_{p}^v(S_v;G_{s,t}^{\infty}(F))\leq 1-\epsilon.
\end{align*}

It follows that there exists $N_v$, such that for any $s\geq N_v$ and $t\geq 1$
\begin{align*}
 &1-\frac{\epsilon}{2}
\geq\sup_{t}\varphi_{p}^v(S_v;G_{s,t}^{\infty}(F))\geq  \varphi_{p}^v(S_v;G_{s,t}^{\infty}(F))
\end{align*}

By Lemma \ref{l73},
\begin{align*}
    \mathbb{P}_{p}(v\xleftrightarrow{W} F;G_{s,t}^{\infty}(F))\leq \sum_{y\in S_v:\partial_V y\cap S_v^{c}\neq \emptyset}\mathbb{P}_{p}(v\xleftrightarrow{S_v^{\circ}}\partial_V y;G_{s,t}^{\infty}(F))\mathbb{P}_{p}(y\xleftrightarrow{W} F;G_{s,t}^{\infty}(F))
\end{align*}
Similarly, there exists a finite $S_y\subseteq V$ satisfying $y\in S_y^{\circ}$ and 
\begin{align*}
\inf_s\sup_t\varphi_{p}^y(S_y;G_{s,t}^{\infty}(F))\leq 1-\epsilon.
\end{align*}
There exists $N_y$, such that for any $s\geq N_y$, $t\geq 1$, we have
\begin{align*}\varphi_{p}^v(S_y;G_{s,t}^{\infty}(F))\leq 1-\frac{\epsilon}{2}
\end{align*}

Again, by Lemma \ref{l73},
\begin{align*}
    \mathbb{P}_{p}(y\xleftrightarrow{W} F;G_{s,t}^{\infty}(F))\leq \sum_{y_1\in S_y:\partial_V y_1\cap S_y^{c}\neq \emptyset}\mathbb{P}_{p}(y\xleftrightarrow{S_v^{\circ}}\partial_V y_1;G_{s,t}^{\infty}(F))\mathbb{P}_{p}(y_1\xleftrightarrow{W} F;G_{s,t}^{\infty}(F))
\end{align*}
Choose 
\begin{align*}
N=\max\{N_v,N_y\}_{y\in S_v:\partial_{v}y\cap S_v^c\neq\emptyset}
\end{align*}
We have for all $s\geq N$, $t\geq N$, 
\begin{align*}
    \mathbb{P}_{p}(v\xleftrightarrow{W} F;G_{s,t}^{\infty}(F))\leq \left(1-\frac{\epsilon}{2}\right)^2.
\end{align*}
Since the graph is locally finite, the process can continue for infinitely many steps. Hence it follows from Corollary \ref{c310} that
\begin{align*}
    \mathbb{P}_{p}(v\xleftrightarrow{W} F)\leq\lim_{m\rightarrow\infty}\left(1-\frac{\epsilon}{2}\right)^m=0.
\end{align*}
Then the lemma follows.
\end{proof}

\bigskip
\noindent\textbf{Proof of Lemma \ref{l311}(2).} We shall use $p_{c,F}$ to denote $p_{c,F}^{site}(G)$ when there is no confusion. 

From the definition of $\tilde{p}_{c,F}$, we see that if $p<\tilde{p}_{c,F}$, there exists $\epsilon_0>0$, $V_{p',\epsilon_0}=\emptyset$. Then \ref{lbc2} follows from Lemma \ref{l37}.
\hfill$\Box$
\\

There is a natural way to couple the percolation process for all $p$ simultaneously. Equip the vertices of $G$ with i.i.d.~random variables $\{U(v)\}_{v\in V}$, uniform in [0,1], and write $\Psi^G$ for the resulting product measure on $[0,1]^V$. For each $p$, the edge set 
\begin{align*}
\{v\in V: U(v)\leq p\}
\end{align*}
has the same distribution as the set of open vertices under $\mathbb{P}_p^{site}$.

Using this standard coupling, one can see that for any $0<p_1\leq p_2\leq 1$, every infinite 1-cluster at level $p_1$ is a subset of an infinite 1-cluster at level $p_2$.

\begin{lemma}\label{l74}Let $G=(V,E)$ be an infinite, locally finite planar graph with a well-separated embedding into the plane. For each $p_2>p_1>\inf_{F\in\mathcal{F}}p_{c,F}^{site}(G)$  there exists $F\in\mathcal{F}$, such that for each $\epsilon>0$, $\Psi^{G}$-a.s.~every infinite 1-cluster $\xi$ at level $p_1$ satisfying
\begin{align}
\ol{\xi}\cap F\neq \emptyset\label{xf}
\end{align}
has infinitely many vertices $w$ satisfying
\begin{align}
\mathbb{P}_{p_2}(w\leftrightarrow F)\geq 1-\left(\frac{1-p_2}{1-p_1}\right)^{1-\epsilon}\label{lbi}
\end{align}
\end{lemma}

\begin{proof}Since $p_1>\inf_{F\in\mathcal{F}}p_{c,F}^{site}(G)$, by Lemma \ref{l37}, there exists $F\in \mathcal{F}$, s.t.
\begin{align*}
\PP_p(\cup_{v\in V}\{v\leftrightarrow F\})=1;\ \forall p\geq p_1.
\end{align*}
In other words, there exists $F\in \mathcal{F}$, such that $p_1>p_{c,F}^{site}(G)$.

Let $\xi$ be an infinite 1-cluster at level $p_1$ satisfying (\ref{xf}). Then $\xi$ must be  contained in an infinite 1-cluster satisfying (\ref{xf}) at level $p_2$ in the standard coupling.

By Lemma \ref{l37}, we have for any $u\in V\setminus V_{p_1,\epsilon}(F)$,
\begin{align*}
\PP_{p_1}(u\xleftrightarrow{V\setminus V_{p_1,\epsilon}(F)} F)=0.
\end{align*}
Assume $|V_{p_1,\epsilon}(F)\cap \xi|<\infty$, then there exists $u\in \xi\setminus V_{p_1,\epsilon}(F)$, such that $u\xleftrightarrow{V\setminus V_{p_1,\epsilon}(F)}F$ occurs at level $p_1$. But (\ref{p10}) implies that this event has probability 0. It follows that almost surely $|V_{p_1,\epsilon}(F)\cap \xi|=\infty$. By Lemma \ref{l311}(1), it follows that for any vertex $w\in V_{p_1,\epsilon}(F)\cap \xi$, (\ref{lbi}) holds.
\end{proof}

\begin{lemma}\label{lem316}Let $G=(V,E)$ be an infinite, locally finite planar graph with a well-separated embedding into the plane. Let $p_2>p_1>p_{c,F}^{site}(G)$ and $\epsilon>0$. Let $\xi$ be an infinite 1-cluster at level $p_1$ such that (\ref{xf}) holds (which must be contained in an infinite 1-cluster at level $p_2$ in the standard coupling). Then for any $\delta>0$, there exist $w_1,\ldots,w_N\in \xi$ satisfying (\ref{lbi2}), such that for any $R\subset V$ satisfying $\{w_1,\ldots,w_N\}\subset R$, one  has
\begin{align*}
\mathbb{P}_{p_2}(R\leftrightarrow F)>1-\delta.
\end{align*}
\end{lemma}

\begin{proof}Let $U_1$ be the set consisting of all the vertices in $\xi$ satisfying (\ref{lbi2}). Then by Lemma \ref{l74}, $|U_1|=\infty$. Let $w_1\in U_1$ be arbitrary, then $w_1$ satisfies (\ref{lbi2}). It follows that
\begin{align}
\PP_{p_2}(w_1\nleftrightarrow F)
&=\sup_s\inf_t \PP_{p_2}(w_1\nleftrightarrow \partial_V D_F^{\frac{1}{s}};G_{s,t}^{\infty}(F))\leq \left(\frac{1-p_2}{1-p_1}\right)^{1-\epsilon}\label{ftn}
\end{align}
Here for each event $E$, $\PP_{p_2}(E;G_{s,t}^{\infty}(F))$ is the probability of $E$ with respect to the i.i.d. Bernoulli($p_2$) site percolation process on the graph $G_{s,t}^{\infty}(F)$.
Then for any $\delta'>0$, there exists $N\geq 1$ s.t.~for any $s\geq N$ 
\begin{align}
\inf_t\PP_{p_2}(w_1\nleftrightarrow \partial_V D_F^{\frac{1}{s}};G_{s,t}^{\infty}(F)) \leq\PP_{p_2}(w_1\nleftrightarrow F)\leq \inf_t\PP_{p_2}(w_1\nleftrightarrow \partial_V D_F^{\frac{1}{s}};G_{s,t}^{\infty}(F))+\delta'.\label{w1d1}
\end{align}

Note that in $\xi\cap \left[D_{F}^{\frac{1}{s}}\setminus \partial_V D_{F}^{\frac{1}{s}}\right]$, there exists an infinite 1-cluster $\xi_1$ at level $p_1$ satisfying
\begin{align*}
\ol{\xi}_1\cap F\neq \emptyset.
\end{align*}

By Lemma \ref{l74}, the set of vertices $U_2\subset \xi_1$ satisfying (\ref{lbi}) in the graph $G_{\left[D_{F}^{\frac{1}{s}}\setminus \partial_V D_{F}^{\frac{1}{s}}\right]}$ (the subgraph of $G$ induced by vertices in $\left[D_{F}^{\frac{1}{s}}\setminus \partial_V D_{F}^{\frac{1}{s}}\right]$) contains infinitely many vertices.
Let $w_2\in U_2$ and $G_1:=G_{\left[D_{F}^{\frac{1}{s}}\setminus \partial_V D_{F}^{\frac{1}{s}}\right]}$. By (\ref{w1d1}), we have 
\begin{align*}
\PP_{p_2}(\{w_1,w_2\}\nleftrightarrow F)
\leq \PP_{p_2}(\cap_t\{w_1\nleftrightarrow \partial_V D_F^{\frac{1}{s}};G_{s,t}^{\infty}(F)\}\cap\{w_2\nleftrightarrow  F\})+\delta'
\end{align*}
It then follows from (\ref{ftn}) that there exists $N_1\geq1$ s.t.~for all $s,s_1\geq N_1$ we have
\begin{align*}
&\PP_{p_2}([\cap_t\{w_1\nleftrightarrow \partial_V D_F^{\frac{1}{s}};G_{s,t}^{\infty}(F)\}]\cap[\cap_{t_1}\{w_1\nleftrightarrow \partial_V D_F^{\frac{1}{s_1}};G_{s_1,t_1}^{\infty}(F)\}])
\leq \PP_{p_2}(\{w_1,w_2\}\nleftrightarrow F)
\\&\leq \PP_{p_2}([\cap_t\{w_1\nleftrightarrow \partial_V D_F^{\frac{1}{s}};G_{s,t}^{\infty}(F)\}]\cap[\cap_{t_1}\{w_1\nleftrightarrow \partial_V D_F^{\frac{1}{s_1}};G_{s_1,t_1}^{\infty}(F)\}])+\delta'\left(1+\frac{1}{2}
\right)
\end{align*}

One can continue this process and obtain that here exists $N_{M-1}\geq 1$, s.t.
\begin{itemize}
\item \begin{align*}
N<N_1<N_2<\ldots< N_{M-1}
\end{align*}
\item $s\geq N,s_1\geq N_1,\ldots,s_{M-1}\geq N_{M-1}$ are arbitrary;
\item Note that for every $s\geq 1$,
\begin{align*}
p_{c,F}^{site}(G)=p_{c,F}^{site}\left(G_{\left[D_{F}^{\frac{1}{s}}\setminus \partial_V D_{F}^{\frac{1}{s}} \right]}\right).
\end{align*}

For each $1\leq j\leq M-1$, $\xi_j$ is the infinite 1-cluster in $\xi_{j-1}\cap D_F^{\frac{1}{s_{j-1}}}$ in the graph $G_{\left[D_{F}^{\frac{1}{s_{j-1}}}\setminus \partial_V D_{F}^{\frac{1}{s_{j-1}}} \right]}$ satisfying
\begin{align*}
\ol{\xi}_j\cap F\neq\emptyset;
\end{align*}    
where $\xi_0=\xi$, $s_0=s$.
\item For each $2\leq j\leq M$, the set of vertices $U_j\subset \xi_1$ satisfying (\ref{lbi}) is infinite; let $w_j\in U_j$
\item For all $s\geq N,s_1\geq N_1,\ldots,s_{M-1}\geq N_{M-1}$ we have
\begin{align}
\label{tc2}&\PP_{p_2}(\cap_{j=0}^{M-1}\cap_{t_j}\{w_{j+1}\nleftrightarrow \partial_V D_F^{\frac{1}{s_j}};G_{s_j,t_j}^{\infty}\})\leq \PP_{p_2}(\{w_1,w_2,\ldots,w_M\}\nleftrightarrow F)\\
&\leq \PP_{p_2}(\cap_{j=0}^{M-1}\cap_{t_j}\{w_{j+1}\nleftrightarrow \partial_V D_F^{\frac{1}{s_j}};G_{s_j,t_j}^{\infty}\})+\delta'\left(\sum_{j=0}^{M-1}\frac{1}{2^j}
\right)\notag
\end{align}
where $t_0=t$.
\end{itemize}
Note that for any $t_0,\ldots,t_{M-1}\geq 1$,
\begin{align*}
\PP_{p_2}(\cap_{j=0}^{M-1}\cap_{t_j}\{w_{j+1}\nleftrightarrow \partial_V D_F^{\frac{1}{s_j}};G_{s_j,t_j}^{\infty}\})
\leq \PP_{p_2}(\cap_{j=0}^{M-1}\{w_{j+1}\nleftrightarrow \partial_V D_F^{\frac{1}{s_j}};G_{s_j,t_j}^{\infty}\});
\end{align*}
and there exists $L_0,\ldots,L_j$ (depend on $s_0,\ldots,s_j$), such that for any $t_j\geq L_j\ (0\leq j\leq M-1)$, and any $0\le m\leq M-1$, we have
\begin{align}
\PP_{p_2}(\cap_{j=0}^{m}\{w_{j+1}\nleftrightarrow \partial_V D_F^{\frac{1}{s_j}};G_{s_j,t_j}^{\infty}\})\leq 
\PP_{p_2}(\cap_{j=0}^{m}\cap_{t_j}\{w_{j+1}\nleftrightarrow \partial_V D_F^{\frac{1}{s_j}};G_{s_j,t_j}^{\infty}\})+\delta'\left(\sum_{j=0}^m\frac{1}{2^j}\right)\label{tc1}
\end{align}
Moreover,
\begin{align}
\label{s20}
&\PP_{p_2}(\{w_1\nleftrightarrow \partial_V D_F^{\frac{1}{s}};G_{s,t}^{\infty}\}\cap \{w_2\nleftrightarrow \partial_V D_F^{\frac{1}{s_1}};G_{s_1,t_1}^{\infty}\})\\
&=\PP_{p_2}(\{w_2\nleftrightarrow \partial_V D_F^{\frac{1}{s_1}};G_{s_1,t_1}^{\infty}\}|\{w_1\nleftrightarrow \partial_V D_F^{\frac{1}{s}};G_{s,t}^{\infty}\})\PP_{p_2}(\{w_1\nleftrightarrow \partial_V D_F^{\frac{1}{s}};G_{s,t}^{\infty}\})\notag\\
&\leq \PP_{p_2}(\{w_2\nleftrightarrow \partial_V D_F^{\frac{1}{s_1}};G_{s_1,t_1}^{\infty}\}|\{w_1\nleftrightarrow \partial_V D_F^{\frac{1}{s}};G_{s,t}^{\infty}\})
\left[\PP_{p_2}(\cap_t\{w_1\nleftrightarrow \partial_V D_F^{\frac{1}{s}};G_{s,t}^{\infty}\})+\delta'\right]\notag\\
&\leq \PP_{p_2}(\{w_2\nleftrightarrow \partial_V D_F^{\frac{1}{s_1}};G_{s_1,t_1}^{\infty}\}|\{w_1\nleftrightarrow \partial_V D_F^{\frac{1}{s}};G_{s,t}^{\infty}\})
\left[\PP_{p_2}(w_1\nleftrightarrow F)+\delta'\right]\notag\\
&\leq \PP_{p_2}(\{w_2\nleftrightarrow \partial_V D_F^{\frac{1}{s_1}};G_{s_1,t_1}^{\infty}\}|\{w_1\nleftrightarrow \partial_V D_F^{\frac{1}{s}};G_{s,t}^{\infty}\})
\left[\left(\frac{1-p_2}{1-p_1}\right)^{1-\epsilon}+\delta'\right]\notag
\end{align}
where the third line follows from (\ref{tc1}), and the fourth line follows from (\ref{tc2}), and the fifth line follows from (\ref{ftn}).
Furthermore, let
\begin{align*}
G_{1,c}:=G_{s,t}^{\infty}\setminus [D_F^{\frac{1}{s}}\setminus \partial_V D_F^{\frac{1}{s}}];\qquad
G_1:=G_{s_1,t_1}^{\infty}\setminus G_{1,c}.
\end{align*}
Then
\begin{align}
\PP_{p_2}(\{w_2\nleftrightarrow \partial_V D_F^{\frac{1}{s_1}};G_{s_1,t_1}^{\infty}\}|\{w_1\nleftrightarrow \partial_V D_F^{\frac{1}{s}};G_{s,t}^{\infty}\})\leq \PP_{p_2}(\{w_2\nleftrightarrow \partial_V D_F^{\frac{1}{s_1}};G_1\}).\label{s21}
\end{align}
Note that $G_{1,c}$ is a finite graph, and therefore we have $p_{c,F}^{site}(G)=p_{c,F}^{site}(G\setminus G_{1,c})<p_1<p_2$. Then
\begin{align}
\PP_{p_2}(w_2\nleftrightarrow F; G\setminus G_{1,c})=
\sup_{s_1}\inf_{t_1}\PP_{p_2}(\{w_2\nleftrightarrow \partial_V D_F^{\frac{1}{s_1}};G_1\})\label{s22}
\end{align}
Since $w_2$ is chosen to be a vertex in $G_{\left[D_{F}^{\frac{1}{s}}\setminus \partial_VD_{F}^{\frac{1}{s}}\right]}$ satisfying (\ref{lbi}), and $G_{\left[D_{F}^{\frac{1}{s}}\setminus \partial_VD_{F}^{\frac{1}{s}}\right]}$ is a subgraph of $G\setminus G_{1,c}$, it follows that 
\begin{align}
\PP_{p_2}(w_2\nleftrightarrow F; G\setminus G_{1,c})\leq \PP_{p_2}\left(w_2\nleftrightarrow F; G_{\left[D_{F}^{\frac{1}{s}}\setminus \partial_VD_{F}^{\frac{1}{s}}\right]}\right)\leq\left(\frac{1-p_2}{1-p_1}\right)^{1-\epsilon}\label{s23}
\end{align}
By (\ref{s22}), (\ref{s23}), we infer that for all $s_1\geq 1$, there exists $\tilde{L}_1\geq 1$ (depend on $s_1$), such that for all $t_1\geq \tilde{L}_1$
\begin{align}
 \PP_{p_2}(\{w_2\nleftrightarrow \partial_V D_F^{\frac{1}{s_1}};G_1\})\le \PP_{p_2}(w_2\nleftrightarrow F;G\setminus G_{1,c})+\delta'\leq \left(\frac{1-p_2}{1-p_1}\right)^{1-\epsilon}+\delta'\label{s24}
\end{align}
It follows from (\ref{s20}) (\ref{s21}) and (\ref{s24}) that
\begin{align*}
\PP_{p_2}(\{w_1\nleftrightarrow \partial_V D_F^{\frac{1}{s}};G_{s,t}^{\infty}\}\cap \{w_2\nleftrightarrow \partial_V D_F^{\frac{1}{s_1}};G_{s_1,t_1}^{\infty}\})\leq \left[\left(\frac{1-p_2}{1-p_1}\right)^{1-\epsilon}+\delta'\right]^2
\end{align*}
We shall continue this process. Suppose
\begin{align}
s<s_1<s_2<\ldots<s_{m}.\label{sic}
\end{align}
Assume we have proved that 
\begin{align}
\PP_{p_2}(\cap_{j=0}^{m-1}\{w_{j+1}\nleftrightarrow \partial_V D_F^{\frac{1}{s_j}};G_{s_j,t_j}^{\infty}\})\leq 
\left(\left(\frac{1-p_2}{1-p_1}\right)^{1-\epsilon}+\delta'\right)^{m}\label{pwm}
\end{align}
Then 
\begin{align}
\label{s20m}
&\PP_{p_2}(\cap_{j=0}^{m}\{w_{j+1}\nleftrightarrow \partial_V D_F^{\frac{1}{s_j}};G_{s_j,t_j}^{\infty}\})\\
&=\PP_{p_2}(\{w_{m+1}\nleftrightarrow \partial_V D_F^{\frac{1}{s_m}};G_{s_m,t_m}^{\infty}\}|\cap_{j=0}^{m-1}\{w_{j+1}\nleftrightarrow \partial_V D_F^{\frac{1}{s_j}};G_{s_j,t_j}^{\infty}\})\PP_{p_2}(\cap_{j=0}^{m-1}\{w_{j+1}\nleftrightarrow \partial_V D_F^{\frac{1}{s_j}};G_{s_j,t_j}^{\infty}\})\notag\\
&\leq \PP_{p_2}(\{w_{m+1}\nleftrightarrow \partial_V D_F^{\frac{1}{s_m}};G_{s_m,t_m}^{\infty}\}|\cap_{j=0}^{m-1}\{w_{j+1}\nleftrightarrow \partial_V D_F^{\frac{1}{s_j}};G_{s_j,t_j}^{\infty}\})\left(\left(\frac{1-p_2}{1-p_1}\right)^{1-\epsilon}+\delta'\right)^{m}\notag
\end{align}
where the third line follows from (\ref{pwm}).

Furthermore, let
\begin{align*}
G_{j,c}:=G_{s_{j-1},t_{j-1}}^{\infty}\setminus [D_F^{\frac{1}{s_{j-1}}}\setminus \partial_V D_F^{\frac{1}{s_{j-1}}}];\ \forall 1\leq j\leq m.
\end{align*}
and
\begin{align*}
G_m:=G_{s_m,t_m}^{\infty}\setminus[\cup_{j=1}^m G_{j,c}].
\end{align*}
Then
\begin{align}
\PP_{p_2}(\{w_{m+1}\nleftrightarrow \partial_V D_F^{\frac{1}{s_m}};G_{s_m,t_m}^{\infty}\}|\cap_{j=0}^{m-1}\{w_{j+1}\nleftrightarrow \partial_V D_F^{\frac{1}{s_j}};G_{s_j,t_j}^{\infty}\})\leq \PP_{p_2}(\{w_{m+1}\nleftrightarrow \partial_V D_F^{\frac{1}{s_1}};G_m\})\label{s21m}
\end{align}
Note that $\cup_{j=1}^mG_{j,c}$ is a finite graph, and therefore we have $p_{c,F}^{site}(G)=p_{c,F}^{site}(G\setminus [\cup_{j=1}^mG_{j,c}])<p_1<p_2$. Then
\begin{align}
\PP_{p_2}(w_{m+1}\nleftrightarrow F; G\setminus [\cup_{j=1}^m G_{j,c}])=
\sup_{s_m}\inf_{t_m}\PP_{p_2}(\{w_{m+1}\nleftrightarrow \partial_V D_F^{\frac{1}{s_m}};G_m\})\label{s22m}
\end{align}
Since $w_{m+1}$ is chosen to be a vertex in $G_{\left[D_{F}^{\frac{1}{s_m}}\setminus \partial_VD_{F}^{\frac{1}{s_m}}\right]}$ satisfying (\ref{lbi}), and $G_{\left[D_{F}^{\frac{1}{s_m}}\setminus \partial_VD_{F}^{\frac{1}{s_m}}\right]}$ is a subgraph of $G\setminus [\cup_{j=1}^mG_{j,c}]$ given (\ref{sic}), it follows that 
\begin{align}
\PP_{p_2}(w_{m+1}\nleftrightarrow F; G\setminus [\cup_{j=1}^mG_{j,c}])\leq \PP_{p_2}\left(w_{m+1}\nleftrightarrow F; G_{\left[D_{F}^{\frac{1}{s_m}}\setminus \partial_VD_{F}^{\frac{1}{s_m}}\right]}\right)\leq\left(\frac{1-p_2}{1-p_1}\right)^{1-\epsilon}\label{s23m}
\end{align}
By (\ref{s22m}), (\ref{s23m}), we infer that for all $s_m\geq 1$, there exists $\tilde{L}_m\geq 1$ (depend on $s_m$), such that for all $t_m\geq \tilde{L}_m$
\begin{align}
 \PP_{p_2}(\{w_{m+1}\nleftrightarrow \partial_V D_F^{\frac{1}{s_m}};G_m\})\le \PP_{p_2}(w_{m+1}\nleftrightarrow F;G\setminus [\cup_{j=1}^mG_{j,c}])+\delta'\leq \left(\frac{1-p_2}{1-p_1}\right)^{1-\epsilon}+\delta'\label{s24m}
\end{align}
It follows from (\ref{s20m}) (\ref{s21m}) and (\ref{s24m}) that
\begin{align}
\PP_{p_2}(\cap_{j=0}^{m}\{w_{j+1}\nleftrightarrow \partial_V D_F^{\frac{1}{s_j}};G_{s_j,t_j}^{\infty}\})\leq \left[\left(\frac{1-p_2}{1-p_1}\right)^{1-\epsilon}+\delta'\right]^{m+1}\label{ts3}
\end{align}
By (\ref{tc2}), (\ref{ts3}) we have
\begin{align*}
\PP_{p_1}(\{w_1,\ldots,w_M\}\nleftrightarrow F)\leq\left[\left(\frac{1-p_2}{1-p_1}\right)^{1-\epsilon}+\delta'\right]^{M}+2\delta' 
\end{align*}
Then the lemma follows by letting $M\rightarrow\infty$ and $\delta'\rightarrow 0$.
\end{proof}

\begin{lemma}\label{l27}
Let $G=(V,E)$ be an infinite, connected, locally finite, planar graph with a Freudenthal embedding in $\mathbb{S}^2$. Let $F\in \mathcal{F}$. Assume 
\begin{align*}
p_c^{site}(G)\le p_{c,F}^{site}(G)<\frac{1}{2}\leq p<1-p_{c,F}^{site}(G)\le 1-p_c^{site}(G). 
\end{align*}
Let $p_1\in \left(p,1-p_{c,F}^{site}(G)\right)$ be fixed. Let $\xi$ be an infinite 0-cluster at level $p_1$ satisfying (\ref{xf}), which must be contained in an infinite 0-cluster at level $p$. Then for any $\epsilon\in(0,1)$ and $k\geq 1$, there exist a finite subset $R$ of vertices in $\xi$ satisfying
\begin{align}
\PP_{p}(\cup_{a\in F}\{\partial_FR\xleftrightarrow{0}F\})>1-\left(\frac{\epsilon}{8k}\right)^{2k};\label{rcif}
\end{align}
s.t.~$R$ induces a connected subgraph $G_{R}$ of $G$, and $i_1<i_2<\ldots<i_{2k}$ satisfying
\begin{align}
\PP_p\!\big(\mathrm{Arm}(i_1,\ldots,i_{2k};R;F)\big)\ \ge\ 1-\eps.\label{alb1}
\end{align}
\end{lemma}

\begin{proof}
Since $p_1<1-p_{c,F}^{site}(G)$, we have $1-p_1>p_{c,F}^{site}(G)$; thus the $0$-phase is supercritical at level $p_1$. By Lemma \ref{lem316}, for any $\epsilon>0$ and $k\geq 1$ in $\xi$ there exists $N$ vertices $w_1,\ldots,w_N$ such that
\begin{align}
\PP_{p}(\{w_1,\ldots,w_N\}\xleftrightarrow{\{\si=0\}}F)\geq 1-\left(\frac{\epsilon}{8k}\right)^{2k}\label{wni}
\end{align}
Let $R\subset V$ satisfy
\begin{itemize}
\item $\{w_1,\ldots,w_N\}\subseteq R\subset \xi$; and
\item the induced subgraph $G_{R}$ of $R$ in $G$ is connected; and
\item $|R|<\infty$.
\end{itemize}
Then (\ref{rcif}) follows from (\ref{wni}). Label all the vertices in $\partial_F R$ by $\{u_1,\ldots,u_M\}$ in cyclic order as in Definition \ref{df26}.

For $1\le i\le j\le M$, let
\[
q_{i,j}:=\PP^{\mathrm{site}}_p\!\Big(\ \{u_i,\ldots,u_j\}\ \xnleftrightarrow{\ [G\setminus R]\cap \{\sigma =0\}\ }\ F\ \Big);
\]
where $G\setminus R$ is the subgraph of $G$ induced by vertices in $V\setminus R$.
Then $q_{1,i}$ is decreasing in $i$. Then (\ref{wni}) implies that
\[
q_{1,M}\ <\ \left(\frac{\epsilon}{8k}\right)^{\,2k}.
\]
Let
\[
i_1:=\min\Big\{i\in\{1,\ldots,M\}: q_{1,i}\le \frac{\epsilon}{4k}\Big\}.
\]
By FKG for decreasing events,
\[q_{1,i_1}\geq q_{1,i_{1}-1}\mathbb{P}_c^{site}(\si(u_{i_1})=1)\geq \frac{\epsilon}{8k}.\]
It follows that
\[
q_{i_1+1,M}\ \le\ \frac{q_{1,M}}{q_{1,i_1}}\ \le\  \left(\frac{\epsilon}{8k}\right)^{\,2k-1}.
\]
Iterating this selection, we find indices $1\le i_1<\cdots<i_{2k}=M$ such that for each $j=1,\ldots,2k$,
\[
q_{\,i_{j-1}+1,\,i_j}\ \le\  \frac{\eps}{4k}.
\]

Define similarly
\[
r_{i,j}:=\PP^{\mathrm{site}}_p\!\Big(\ \{u_i,\ldots,u_j\}\ \xnleftrightarrow{\ [G\setminus R]\cap \{\sigma=1\}\ }F\ \Big),
\]
By stochastic monotonicity, we have for each $j$,
\[
r_{\,i_{j-1}+1,\,i_j}\ \le\ \frac{\eps}{4k},
\]
A union bound over the $4k$ arm events then yields (\ref{alb1}).
\end{proof}

\section{Infinitely many infinite clusters in the supercritical half}
\label{s5}

We now combine the geometric arm events from Section~\ref{sec4} with the end--directed
connectivity estimates from Section~\ref{sec5}.
The core argument is an exploration of a supercritical $0$--cluster at a slightly larger parameter,
followed by a conditional arm construction in the unexplored complement.
Topological separation on $\mathbb S^2$ enforces conditional independence across suitably separated
regions, and an amplification step then yields arbitrarily many disjoint infinite $1$--clusters.
This proves the main non--uniqueness statement in the supercritical half of the coexistence picture
under the corresponding hypothesis on $p^{\mathrm{site}}_{c,F}(G)$ (and in particular in the countable
$\mathcal F(G)$ regime).

We now describe an exploration process to construct a 0-cluster at a vertex $v_0\in V$ at level $p_1$. Let $\si\in\{0,1\}^V$ be a random site percolation configuration on $G$ at level $p_1$. In the process we will generate a random sequence of two finite subsets of $V$, $C_i^{p_1}(v_0)$ and $\partial C_i^{p_1}(v_0)$, $i=0,1,\ldots$, with the first one growing the 0-cluster at $v_0$ and the second one growing towards the boundary of this cluster. For this purpose we order the vertices in $V$ in an arbitrary way. If $\si(v_0)=1$ set
\begin{align*}
C_0^{p_1}(v_0)=\emptyset;\ \mathrm{and}\ \partial C_0^{p_1}(v_0)=\emptyset.
\end{align*}
Then recursively proceed as follows. In the $i$-th step ask whether there is any vertex which is at distance 1 in $G$ from $C_{i-1}^{p_1}(v_0)$ but not in $\partial C_{i-1}^{p_1}(v_0)$. In case no such vertex exists, the construction has been completed, our sequences are finite ones $C^{p_1}(v_0)=C_{i-1}^{p_1}(v_0)$. Otherwise let $v_i$ be the first such vertex. If $\sigma(v_i)=1$, set
\begin{align*}
C_i^{p_1}(v_0)=C_{i-1}^{p_1}(v_0);\ \ \mathrm{and}\ \partial C_i^{p_1}(v_0)=
\partial C_{i-1}^{p_1}(v_0)\cup\{v_i\}.
\end{align*}
If $\sigma(v_i)=0$, set
\begin{align*}
C_i^{p_1}(v_0)=C_{i-1}^{p_1}(v_0)\cup \{v_i\};\ \ \mathrm{and}\ \partial C_i^{p_1}(v_0)=
\partial C_{i-1}^{p_1}(v_0).
\end{align*}
Note that the vertices which will eventually be contained in $C_i^{p_1}(v_0)$ (resp.\ $\partial C_i^{p_1}(v_0)$) are precisely the vertices of $C^{p_1}(v_0)$ (resp.\ $\partial C^{p_1}(v_0)$); here $C^{p_1}(v_0)$ is the 0-cluster at $v_0$ at level $p_1$ and $\partial C^{p_1}(v_0)$ consist of all the vertices not in $C^{p_1}(v_0)$ but has distance 1 with $C^{p_1}(v_0)$ in the graph $G$.

It is straightforward to verify the following lemma by planarity.
\begin{lemma}[Alternating--arm separation at a fixed end equivalent class]\label{lem:sep}
Let $G$ be a locally finite planar graph equipped with a well--separated
Freudenthal embedding $\phi\colon G\hookrightarrow\mathbb S^2$, and let
$\widehat\phi\colon|G|\to\mathbb S^2$ denote the induced embedding of the
Freudenthal compactification so that $\widehat\phi(\Omega(G))=\mathcal{A}\subset\mathbb S^2$
is the set of accumulation points (ends). Let $R\subset V$ be finite and connected.
Fix an end equivalence class $F\in \mathcal{F}$.

Suppose there exist $2k$ boundary subarcs
$\mathrm{Arc}_1,\ldots,\mathrm{Arc}_{2k}\subset\partial_{F}R$ and an
alternating--arm configuration in $R^{\mathrm c}$ with the following
properties:
\begin{itemize}
\item For every odd $j$, there is a $1$--arm $\gamma^{(1)}_j$ starting from
$\mathrm{Arc}_j$ and running to infinity in $R^{\mathrm c}$;
\item For every even $j$, there is a $0$--arm $\gamma^{(0)}_j$ starting from
$\mathrm{Arc}_j$ such that $\widehat\phi(\gamma^{(0)}_j)$ converges to the same end $F$.
\end{itemize}
Let $t(F)$ be the number of even indices $j$ for which $\gamma^{(0)}_j$ tends to $F$.
Then the curves $\phi(R)\cup \widehat\phi(\gamma^{(0)}_{j_1})\cup
\widehat\phi(\gamma^{(0)}_{j_2})$ with $j_1<j_2$ consecutive in this collection
divide $G\setminus [R\cup[\cup_{j}\gamma_{2j}^{(0)}]]$ into pairwise disjoint subgraphs, and each subgraph contains the image of
some odd arc $\mathrm{Arc}_{\ell}$ together with its $1$--arm $\gamma^{(1)}_{\ell}$.
Consequently, $R^{\mathrm c}$ contains at least $t(F)$
pairwise distinct infinite $1$--clusters.
\end{lemma}

\begin{theorem}\label{thm:main}
Let $G=(V,E)$ be an infinite, connected, locally finite graph with a well-separated Freudenthal embedding in the plane. If $\inf_{F\in\mathcal{F}} p_{c,F}^{site}(G)<\tfrac12$, then for any $p\in\big(\tfrac12,\,1-\inf_{F\in\mathcal{F}} p_{c,F}^{site}(G)\big)$, then $\mathbb{P}_{p}^{site}$-a.s.~there are infinitely many infinite 1-clusters.
\end{theorem}

\begin{proof}Let $p_1\in \left(p,1-\inf_{F\in\mathcal{F}} p_{c,F}^{site}(G)\right)$. 
Since $1-p_1>\inf_{F\in\mathcal{F}} p_{c,F}^{site}(G)$, there exists $F\in \mathcal{F}$, such that $1-p_1>p_{c,F}^{site}(G)$. Then there exists $x\in V$, such that
\begin{align*}
 \PP_{p_1}(x\xleftrightarrow{G\cap \{\si=0\}} F)>0.
\end{align*}

Let $\xi_x$ be the infinite 0-cluster at $x$ at level $p_1$ satisfying
\begin{align*}
\ol{\xi}_{x}\cap F\neq \emptyset 
\end{align*}

By Lemma \ref{l27}, we have for any $\epsilon>0$ and any positive integer $k$, there exists a finite set of vertices $R\subset \xi_{x}$ such that 
\begin{align}
\PP^{\mathrm{site}}_p\!\big(\mathrm{Arm}(i_1,\ldots,i_{2k};R;F)\big)\ \ge\ 1-\eps,\label{pg1}
\end{align}
for some $i_1,\ldots,i_{2k}$.
In particular if the event $\mathrm{Arm}(i_1,\ldots,i_{2k};R;F)$ occurs, then the number of\ infinite 0-clusters, whose closure intersect $F$, at level $p$ in 
$\mathbb{S}^2\setminus R$ intersecting $\partial_F R$ is at least $k$, and the number of infinite 1-clusters at level $p$, whose closure intersect $F$, in\ $\mathbb{S}^2\setminus R$ intersecting $\partial_F R$ is at least $k$.

Explore the 0-cluster at $x$ at level $p_1$ until the first time $\tau$ for which $R\subset C_{\tau}^{p_1}(x)$. By Lemma \ref{l27} and the exploration process described above, a.s.~$\tau<\infty$.

 Let $\tilde{\partial} C_{\tau}^{p_1}(x)$ be the set of vertices in $V\setminus C_{\tau}^{p_1}(x)$ that are adjacent to $C_{\tau}^{p_1}(x)$ in $G_*$. Since the exploration process up to time $\tau$ gives no information about the set of vertices in $V\setminus [C_{\tau}^{p_1}(x)\cup\tilde{\partial} C_{\tau}^{p_1}(x)]$, by (\ref{pg1}) we have
\begin{align*}
&\PP_p(\mathrm{Arm}(j_1,\ldots,j_{2k};C_{\tau}^{p_1}(x)\cup \tilde{\partial} C_{\tau}^{p_1}(x);F)\ \mathrm{for\ some}\ j_1,\ldots,j_{2k}\\
&|\mathrm{the\ exploration\ process\ up\ to\ time\ }\tau)\geq 1-\epsilon.
\end{align*}
 In particular, if $\mathrm{Arm}(j_1,\ldots,j_{2k};C_{\tau}^{p_1}(x)\cup \tilde{\partial} C_{\tau}^{p_1}(x);F)$ occurs, 
the the percolation process at level $p$ restricted to the vertex set $V\setminus [C_{\tau}^{p_1}(x)\cup \tilde{\partial} C_{\tau}^{p_1}(x)]$ has at least $k$ infinite 0-clusters that
 contain some vertex with distance 1 in $G$ to $\tilde{\partial} C_{\tau}^{p_1}(v_0)$; and at least $k$ distinct infinite 1-clusters restricted to the vertex set $V\setminus [C_{\tau}^{p_1}(x)\cup \tilde{\partial} C_{\tau}^{p_1}(x)]$.

If $x$ is in an infinite 0-cluster at level $p_1$, then 
\begin{align}
U(w)> p_1;\ \forall w\in C_{\tau}^{p_1}(x)\label{uwl}
\end{align}
Given (\ref{uwl}) and the exploration process up to time $\tau$, each $v\in \tilde{\partial} C_{\tau}^{p_1}(x)$ is closed at level $p$ independently with conditional probability at least $\frac{p_1-p}{p_1}$. If $\mathrm{Arm}(j_1,\ldots,j_{2k};C_{\tau}^{p_1}(x)\cup \tilde{\partial} C_{\tau}^{p_1}(x);F)$ occurs, one may pick $u_1,\ldots,u_k\in \tilde{\partial} C_{\tau}^{p_1}(x)$ adjacent to, in $G$, $k$ different 0-clusters at level $p$ in $V\setminus[C_{\tau}^{p_1}(v_0)\cup \tilde{\partial} C_{\tau}^{p_1}(v_0)]$ whose closure intersect $F$; such that
\begin{itemize}
\item for each $1\leq t\leq k$, $u_t$  is adjacent to $\mathrm{Arc}_{2t}$ of $\partial_F[C_{\tau}^{p_1}(x)\cup \tilde{\partial} C_{\tau}^{p_1}(x)]$
\item for each $1\leq t\leq k$,
$\mathrm{Arc}_{2t-1}\xleftrightarrow {[C_{\tau}^{p_1}(x)\cup \tilde{\partial} C_{\tau}^{p_1}(x)]^c\cap \{\sigma=1\}}F$ occurs.
\end{itemize}

From the exploration process we see that the state of each $u_i (1\leq i\leq k)$ is either undetermined or open at level $p_1$. We have
\begin{itemize}
\item If the state of $u_i$ is undetermined at level $p_1$, then at level $p$, with probability $1-p$ $u_i$ is closed;
\item If the state of $u_i$ is open at level $p_1$, then at level $p$, with probability $\frac{p_1-p}{p_1}$ $u_i$ is closed by the standard coupling.
\end{itemize}
In either case with probability at least $\frac{p_1-p}{p_1}$ $u_i$ is closed at level $p$.

If $x$ is in an infinite 0-cluster at level $p_1$ (hence must be in an infinite 0-cluster at level $p$) and at least $l$ of the vertices $u_1,\ldots,u_k$ are closed at level $p$, then there are at least $l$ distinct infinite 1-clusters at level $p$. Let $N_1$ be the total number of infinite 1-clusters. Then
\begin{align*}
\PP_p(N_1<l)\leq \epsilon+\sum_{j=0}^{l-1}{k\choose j}
\left(\frac{p_1-p}{p_1}\right)^j
\left(\frac{p}{p_1}\right)^{k-j}
\end{align*}
For each fixed $l$, letting $\epsilon\rightarrow 0$ and $k\rightarrow\infty$, we obtain 
\begin{align*}
\PP_p(N_1<l)=0
\end{align*}
 Then the theorem follows.
\end{proof}

\section{Uncountably many end--equivalence classes: the positive regime}
\label{s6}

In this section we address the regime in which $\mathcal F(G)$ is uncountable.
We focus on Case~(2a) in Section~\ref{sec3}, where the percolation parameter lies in a range that forces many ends to remain
inactive for infinite $1$--clusters while infinite 1-clusters still occurs in the complement of neighborhoods of the
active set.
Assuming toward contradiction that the number of infinite $1$--clusters is almost surely finite, we use
compactness of the relevant active set of ends together with a boundary--avoidance argument to produce
an infinite $1$--cluster accumulating at an end that is provably inactive, yielding a contradiction.
Consequently, in this positive regime one still has almost surely infinitely many infinite $1$--clusters.

Recall that $\mathcal{A}$ is the set consisting of all the accumulation points of an embedding of $G$ in $\mathbb{S}^2$. Under the Freudenthal embedding of the graph in $\mathbb{S}^2$, we can identify each end with an accumulation point.
Under the topology defined on the ends of the graph (see Definition \ref{df21}), $\mathcal{A}$ is a compact Hausdorff space; given that $G$ is infinite, connected and locally finite (see also Page 76 of \cite{bs96}).

\begin{lemma}\label{le78}Let $\mathcal{A}_1\subset \mathcal{A}$ consist of all the $a\in \mathcal{A}$ satisfying (\ref{ccds}).
Assume (\ref{fi1}) holds.
then $\mathcal{A}_1$ is a closed subset of $\mathcal{A}$; hence
$\mathcal{A}_1$ is a compact Hausdorff space.
\end{lemma}

\begin{proof}It suffices to show that for any $a\in \overline{\mathcal{A}}_1$, we have $a\in \mathcal{A}_1$.

Assume there exists $a\in \overline{\mathcal{A}}_1\setminus \mathcal{A}_1$. Under the Freudenthal embedding of $G$ into $\mathbb{S}^2$, the topology on ends of the graph in Definition \ref{df21} coincides with the subspace topology on $\mathcal{A}$ induced by the topology on $\mathbb{S}^2$; see Proposition \ref{th52}. In this sence $\mathcal{A}$ is a closed metric space ($T_4$ space).
Then there exists a sequence $\{a_n\}_{n=1}^{\infty}\subset \mathcal{A}_1$, such that
\begin{align}
\lim_{n\rightarrow\infty}a_n=a;\label{sq1}
\end{align}
and 
\begin{align}
\PP_p(\cup_{v\in V}\{v\xleftrightarrow{1}a_n\})=1;\ \forall n\geq 1.\label{sq2}
\end{align}
For each $v\in V$, let $C_v$ be the 1-cluster at vertex $v$. Then (\ref{sq1}) and (\ref{sq2}) imply that
\begin{align*}
\PP_p(a\in \overline{\cup_v {C_v}})=1
\end{align*}
If $1\leq \mathcal{N}_{\infty}<\infty$, we have $\overline{\cup_v {C_v}}=\cup_{v\in V}\overline{C_v}$. Then when (\ref{fi1}) holds, a.s.~there exists $v\in V$, such that $a\in \overline{C}_v$. By Lemma \ref{le76}, a.s.~there exists $v\in V$, such that $v\xleftrightarrow{1}a$ occurs. Then (\ref{ccds}) holds; and therefore $a\in \mathcal{A}_1$.
\end{proof}

\begin{lemma}\label{lm79}Let 
\begin{align}
\max\left\{\frac{1}{2},1-\inf_{F\in\mathcal{F}}p_{c,F}^{site}(G)\right\}\leq p<1.\label{pdm}
\end{align}
Consider case (2a) in Section~\ref{sec3}.
 Assume that (\ref{fi1}) holds. Then there exists $\epsilon>0$ ,  such that all the following hold:
\begin{enumerate}[label=(\Alph*)]
\item 
$G$ has uncountably many ends in 
\begin{align}
\{\mathbb{S}^2\setminus [\cup_{a\in \mathcal{A}_1}B_{\mathbb{S}^2}(a,\epsilon)]\};\label{sa3}
\end{align}
and 
\item to each end $a$ in (\ref{sa3}), 
\begin{align}
\PP_p(\cup_{v\in V}\{v\xleftrightarrow{1} a\})=0.\label{ccd3}
\end{align}

\item In the graph $G\setminus [\cup_{a\in\mathcal{A}_1}B_{\mathbb{S}^2}(a,\epsilon)]$, a.s.~there are infinite 1-clusters.
\end{enumerate}
\end{lemma}

\begin{proof}

Suppose that for any $\epsilon>0$ there are countably many ends in (\ref{sa3}). Then for any positive integers $k\geq 1$, $G$ has countably many ends in
\begin{align}
 \left\{\mathbb{S}^2\setminus \left[\cup_{a\in\mathcal{A}_1}B_{\mathbb{S}^2}\left(a,\frac{1}{k}\right)\right]\right\};\label{ss2}
\end{align}

Under the Freudenthal embedding, each end is an accumulation point.
By Lemma \ref{le78}, when (\ref{fi1}) holds, the set $\mathcal{A}_1$ of all the ends satisfying (\ref{ccds}) is a compact subset of $\mathbb{S}^2$. Therefore, for any  
\begin{align}
a\in\mathcal{A}\setminus \mathcal{A}_1,\label{fcae}
\end{align}
there exists a positive integer $m$, such that
\begin{align*}
a\in \mathbb{S}^2\setminus \left[\cup_{b\in\mathcal{A}_1}B_{\mathbb{S}^2}\left(b,\frac{1}{m}\right)\right]
\end{align*}
Then all the end equivalence classes satisfying (\ref{fcae}) must be contained in 
\begin{align}
\cup_{k=1}^{\infty} \left\{\mathbb{S}^2\setminus \left[\cup_{b\in\mathcal{A}_1}B_{\mathbb{S}^2}\left(b,\frac{1}{k}\right)\right]\right\}\label{sd3}
\end{align}
By (\ref{ss2}), there are countably many ends in (\ref{sd3}); hence there are countably many ends satisfying (\ref{fcae}). Moreover, for any end satisfying (\ref{fcae}), (\ref{ccd3}) holds. Then (\ref{ccd3}), together with the countability of $\mathcal{A}\setminus \mathcal{A}_1$, implies that
\begin{align*}
\PP_p(\cup_{a\in \mathcal{A}\setminus \mathcal{A}_1 }\cup_{v\in V}\{v\xleftrightarrow{1} a\})=0.
\end{align*}
which contradicts (\ref{sst}). Then there exists $\epsilon_0> 0$, such that for all $\epsilon\leq \epsilon_0$, (\ref{sa3}) holds.

(B) is straightforward.

Now we prove (C). For each $k\geq 1$, define an event
\begin{align*}
E_{k}:=\left\{\mathrm{there\ are\ infinite\ 1-clusters\ in\ the\ graph}\ G\setminus \left[\cup_{b\in \mathcal{A}_1}B_{\mathbb{S}^2}\left(b,\frac{1}{k}\right)\right]\right\}
\end{align*}
Then we have
\begin{align*}
E_k\subseteq E_{k+1}.
\end{align*}

By (\ref{sst}), 
\begin{align*}
\lim_{k\rightarrow\infty}\PP_p(E_k)=\PP_p\left(\cup_{k=1}^{\infty}E_{k}\right)=1.
\end{align*}
Therefore there exists $k_1\geq 1$, s.t. $\PP_p(E_{k_1})>\frac{1}{2}$. Since $E_{k_1}$ is tail measurable, it follows that $\PP_p(E_{k_1})=1$. 
In particular there exists $\epsilon_1:=\frac{1}{k_1}$, such that for any $\epsilon\leq \epsilon_1$, (C) holds.

Choose $\epsilon =\min\{\epsilon_0,\epsilon_1\}$, we have (A) and (C) hold with the same $\epsilon$.
\end{proof}

\begin{lemma}\label{le7a}Suppose that (\ref{pdm}) and (\ref{fi1}) holds.
Consider case (2a).
Let $\mathcal{A}_1$ be all the ends satisfying (\ref{ccds}). Then there exists an open set $\mathcal{O}\supset \mathcal{A}_1$ such that 
\begin{enumerate}[label=(\Alph*)]
\item in
\begin{align}
\mathcal{A}\cap \{\mathbb{S}^2\setminus \mathcal{O}\};\label{ss3}
\end{align}
$G$ has uncountably many ends; and 
\item For each end $a$ in (\ref{ss3}), 
\begin{align}
\PP_p(\cup_{v\in V}\{v\xleftrightarrow{1} a\})=0.\label{ccd2}
\end{align}
\item In the graph $G\setminus \mathcal{O}$, a.s.~there are infinite 1-clusters.
\item Let $\partial{\mathcal{O}}:=\overline{\mathcal{O}}\setminus \mathcal{O}$, then $\partial \mathcal{O}\cap \mathcal{A}=\emptyset$.
\end{enumerate}
\end{lemma}

\begin{proof}By Lemma \ref{lm79}, there exists $\epsilon>0$, such that (A)(B)(C) holds with $\mathcal{O}$ replaced by 
\begin{align*}
\cup_{b\in\mathcal{A}_1} B_{\mathbb{S}^2}\left(b,\epsilon\right)
\end{align*}
By Lemma \ref{le78}, when (\ref{fi1}) holds, $\mathcal{A}_1$ is compact, and $\{B_{\mathbb{S}^2}\left(b,\epsilon\right)\}_{b\in \mathcal{A}_1}$ form an open cover of $\mathcal{A}_1$, there exists a finite positive integer $M$ and $z_1,\ldots,z_M\in \mathcal{A}_1$, s.t.
\begin{align*}
\mathcal{A}_1\subset \cup_{j=1}^M B_{\mathbb{S}^2}(z_j,\epsilon).
\end{align*}
When $\mathcal{O}$ is replaced by $\cup_{j=1}^M B_{\mathbb{S}^2}(z_j,\epsilon)$, (A) (B) (C) still hold.

If  $[\partial\cup_{j=1}^M B_{\mathbb{S}^2}(z_j,\epsilon)]\cap \mathcal{A}=\emptyset$; let $\mathcal{O}:=\cup_{j=1}^M B_{\mathbb{S}^2}(a_j,\epsilon)$, then (D) holds.

If  $A_0:=\{[\partial\cup_{j=1}^M B_{\mathbb{S}^2}(z_j,\epsilon)]\cap \mathcal{A}\}\neq \emptyset$, since both $\partial\cup_{j=1}^M B_{\mathbb{S}^2}(z_j,\epsilon)$ and $\mathcal{A}$ are compact, their intersection $A_0$ is compact as well. For each $a\in A_0$ and $b\in\mathcal{A}_1$ there exists a finite set  $K_{a,b}$ of vertices such that 
\begin{itemize}
\item $a(K_{a,b})$ and $b(K_{a,b})$ are two distinct infinite components of $G\setminus K_{a,b}$. 
\end{itemize}

By Lemma \ref{lm27}, the Freudenthal embedding is a homeomorphism from $|G|$ to a subset of $\mathbb{S}^2$, where the latter has subspace topology. Hence, the component $b(K_{a,b})$ lies in an open set $T_{a,b}$ of $\mathbb{S}^2$ and $a\in R_{a,b}:=\mathbb{S}^2\setminus \overline{T_{a,b}}$ and $\partial T_{a,b}\cap \mathcal{A}=\emptyset$. Then 
\begin{align*}
\mathcal{A}_1\subset \cup_{b\in\mathcal{A}_1}T_{a,b}
\end{align*}
In other words, $T_{a,b} (b\in\mathcal{A}_1)$ form an open cover of $\mathcal{A}_1$. By Lemma \ref{le78}, when (\ref{fi1}) holds, $\mathcal{A}_1$ is compact we can find a finite positive integer $K\geq 1$, $\{b_1,\ldots,b_K\}\subset\mathcal{A}_1$, such that
\begin{align*}
\mathcal{A}_1\subset \cup_{j=1}^K T_{a,b_j}
\end{align*}
Let 
\begin{align*}
R_a=\cap_{j=1}^K R_{a,b_j};\qquad T_a=\cup_{j=1}^K T_{a,b_j}
\end{align*}
then both $R_a$ and $T_a$ are open. 
Then $\{R_a\}_{a\in A_0}$ form an open cover of $A_0$, since $A_0$ is compact, there exists a finite positive integer $L$ and
\begin{align*}
x_1,\ldots,x_L\in A_0
\end{align*}
such that 
\begin{align*}
A_0\subset \cup_{i=1}^L R_{x_i}.
\end{align*}
Let 
\begin{align*}
\mathcal{O}=[\cap_{i=1}^{L}T_{x_i}]\cap [\cup_{j=1}^M B_{\mathbb{S}^2}(a_j,\epsilon)];
\end{align*}
then the lemma follows. In particular $\partial\mathcal{O}\cap \mathcal{A}=\emptyset$.
\end{proof}

\begin{lemma}Suppose that (\ref{pdm}) holds. Consider case (2a) in Section~\ref{sec3}; then a.s.~there are infinitely many infinite 1-clusters.
\end{lemma}

\begin{proof}
Let $\mathcal{N}_{\infty}$ be the number of infinite 1-clusters. The event that $\{\mathcal{N}_{\infty}=\infty\}$ is tail-measurable; one has $\PP_p(\mathcal{N}_{\infty}=\infty)\in\{0,1\}$.

Suppose that the conclusion of the proposition does not hold, then $\PP_p(\mathcal{N}_{\infty}=\infty)=0$. Given that $p>\pcs(G)$, we have $\PP_p(1\leq \mathcal{N}_{\infty}<\infty)=1$. Since $p\in(0,1)$ by finite energy we have $\PP_p(\mathcal{N}_{\infty}=1)>0$.

By Lemma \ref{le78}, the set $\mathcal{A}_1$ is a compact metric space.
Let $\mathcal{O}$ be given as in Lemma \ref{le7a}. By Lemma \ref{le7a}(C), in the graph $G\setminus \mathcal{O}$, a.s. percolation occurs. By Lemma \ref{le7a}(D), $\partial \mathcal{O}\cap \mathcal{A}=\emptyset$, and $\partial{\mathcal{O}}$ is compact, it follows that $\partial \mathcal{O}$ intersects at most finitely many vertices and edges of $G$. Since in $G$ a.s.~there are finitely many infinite 1-clusters,
in $G\setminus \mathcal{O}$ a.s.~there are finitely many infinite 1-clusters. 

Then by Lemma \ref{le76} there exists an end $a\in \mathbb{S}^2\setminus \mathcal{O}$, such that
\begin{align*}
\PP_p(\cup_{v\in V\cap [G\setminus \mathcal{O}]}v\xleftrightarrow{[G\setminus\mathcal{O}]\cap \{\sigma=1\}}a)=1.
\end{align*}
Then 
\begin{align*}
\PP_p(\cup_{v\in V}v\xleftrightarrow{\{\sigma=1\}}a)=1.
\end{align*}
But this contradicts Lemma \ref{le7a}(B).
\end{proof}

\section{Counterexamples and Sharpness}\label{s8}

We conclude by showing that the coexistence picture cannot be extended without additional hypotheses.
First we construct, in the uncountable end--equivalence regime (Case~(2b) in Section~\ref{sec3}), an explicit infinite,
connected, locally finite planar graph for which the inequality $p^{\mathrm{site}}_u(G)\ge 1-p^{\mathrm{site}}_c(G)$
fails; equivalently, there exists $p\in(1/2,\,1-p^{\mathrm{site}}_c(G))$ for which only finitely many infinite
$1$--clusters occur $\PP_p$-a.s.
We then give a separate planar example that is not locally finite and in which Benjamini--Schramm type
coexistence predictions break down even more strongly, demonstrating that local finiteness is indispensable
for statements in Conjectures \ref{c7bs} and \ref{c8bs}.

\subsection{Case (2b) in Section~\ref{sec3}.}

 We shall construct an explicit example of an infinite, connected, locally finite graph $G$ for which $p_c^{site}(G)<\frac{1}{2}$ and $p_u^{site}<1-p_c^{site}(G)$.

\begin{construction}[Two--sided $M$--adic tree of strips with triangulated faces]\label{c75}
Let $d\ge 3$ and $M\geq2$ be integers. Set $B:=M^d$.
Let $T_B$ be the rooted $B$--ary tree: the root has degree $B$, and every other vertex
has degree $B+1$ (one parent and $B$ children). For $n\ge 0$ let $L_n$ denote the set of
vertices at depth $n$ (so $|L_n|=B^n=M^{dn}$).

\smallskip
\noindent\textbf{Step 1 (planar level embedding and order).}
Embed $T_B$ in the upper half plane so that all vertices of $L_n$ lie on a horizontal line $\ell_n$,
and so that for each vertex $v\in L_{n-1}$ its $B$ children appear consecutively on $\ell_n$; and all the accumulation points are along the real line.
Moreover, require that these consecutive child--blocks are ordered from left to right
according to the left--to--right order of the parents on $\ell_{n-1}$.
This induces a canonical left--to--right order on $L_n$; we label
\[
L_n=\{(n,i):1\le i\le M^{dn}\},
\]
where $(n,1),(n,2),\dots,(n,M^{dn})$ are listed from left to right along $\ell_n$.

\smallskip
\noindent\textbf{Step 2 ($M$--adic grouping).}
Write $\Sigma:=\{0,1,2,\ldots,M-1\}$ and for $n\ge 1$ set
\[
s_n:=M^{(d-1)n}.
\]
For a word $\mathbf b=(b_1,\dots,b_n)\in \Sigma^n$, define its base--$M$ address
\[
\addr(\mathbf b):=\sum_{t=1}^n b_t\,M^{\,n-t}\in\{0,1,\dots,M^n-1\},
\]
and define the corresponding \emph{group/block}
\[
Q_{\mathbf b}
:=\Bigl\{(n,j):\addr(\mathbf b)\,s_n+1\le j\le (\addr(\mathbf b)+1)\,s_n\Bigr\}\subseteq L_n.
\]
Then $(Q_{\mathbf b})_{\mathbf b\in\Sigma^n}$ partitions $L_n$ into $M^n$ contiguous blocks,
each of size $s_n$.

\smallskip
\noindent\textbf{(Compatibility with the tree structure).}
For every prefix $\mathbf b\in\Sigma^{n-1}$, the set of children (in $L_n$) of vertices in $Q_{\mathbf b}$
is exactly
\[
\bigcup_{r\in\Sigma} Q_{(\mathbf b,r)},
\]
and these $M$ blocks appear consecutively from left to right in increasing order of $r=0,1,\ldots,M-1$.

\smallskip
\noindent\textbf{Step 3 (horizontal strip edges and triangulation: first copy).}
Starting from the embedded tree, add \emph{horizontal} edges inside each group:
for every $n\ge 1$, every $\mathbf b\in\Sigma^n$, and every integer
\[
\addr(\mathbf b)\,s_n+1\le j\le (\addr(\mathbf b)+1)\,s_n-1,
\]
add the edge $\{(n,j),(n,j+1)\}$.
Let $G_{\mathrm{strip}}$ be the resulting embedded planar graph.

Triangulate every finite face $f$ of $G_{\mathrm{strip}}$ whose boundary cycle has length $\ge 4$
by a deterministic fan rule:
let $w(f)$ be an uppermost boundary vertex, and among those choose the \emph{leftmost} one.
Add all diagonals inside $f$ from $w(f)$ to every other boundary vertex of $f$ not already adjacent to $w(f)$.
This produces a triangulation of each finite face without crossings.
Denote the resulting graph by $\widetilde G$.

\smallskip
\noindent\textbf{Step 4 (doubling and vertical gluing).}
Let $\widetilde G^{(1)}$ and $\widetilde G^{(2)}$ be two disjoint copies of $\widetilde G$,
such that $\widetilde G^{(1)}$ has the same embedding as $\widetilde G$ in the upper half plane, and
 $\tilde{G}^{(2)}$ is the reflection of $\tilde{G}^{(1)}$ with resepect to the real axis.

For each $n\ge 1$ and each $1\le i\le M^n-1$, define the left and right endpoints of the $i$-th block at level $n$ by
\[
\ell_{n,i}:=(n,(i-1)s_n+1),
\qquad
r_{n,i}:=(n,is_n).
\]
Add \emph{vertical} edges joining the corresponding vertices across the two copies:
\[
\{\ell_{n,i}^{(1)},\ell_{n,i}^{(2)}\}
\quad\text{and}\quad
\{r_{n,i}^{(1)},r_{n,i}^{(2)}\}
\qquad (n\ge 1,\ 1\le i\le M^n-1).
\]
Let $G_0$ denote the resulting planar graph.

\smallskip
\noindent\textbf{Step 5 (final triangulation).}
Finally, triangulate every finite face of $G_0$ using the same fan rule as in Step~3. 
Let $G$ be the resulting graph.
\end{construction}

\begin{remark}\label{c75r}
The graph $G$ is infinite, connected, locally finite, and planar; moreover, every finite face of $G$ is a triangle.
In addition, the fan triangulation rule implies that each vertex at depth $n\ge1$ has at most two neighbours
at depth $n-1$ (its tree-parent and at most one further depth-$(n-1)$ neighbour created by a fan diagonal).
\end{remark}



\begin{theorem}\label{t76}There exist a positive integer $d\geq 3$, $M\geq 17$ such that in the graph $G$ constructed as in Construction \ref{c75},
\begin{enumerate}
\item $\PP_p$-a.s.~there are finitely many infinite 1-clusters for some $p\in \left(\frac{1}{2},1-\pcs(G)\right)$; and
\item $p_u^{site}<1-p_c^{site}(G)$.
\end{enumerate}
\end{theorem}

In the rest of this section we prove Theorem~\ref{t76}.

Let $\mathcal{B}$ be the collection of all infinite sequences where each digit is a number in $\Sigma$. We can see that
\begin{itemize}
\item Each class of equivalent ends in $G$ contains exactly one end; and
\item Under the Freudenthal embedding $\phi:|G|\hookrightarrow \mathbb S^2$, the map
$\Omega(G)\to \mathcal{A}(\phi)$ is a bijection; combined with the natural identification
$\Omega(G)\cong\mathcal B$, this yields a canonical parametrization
$\mathcal{A}(\phi)\cong\mathcal B$.
\end{itemize}

For each $\mathbf{b}=(b_1,b_2,\ldots)\in\mathcal{B}$, 
we construct a subgraph $G_{\mathbf{b}}$ of $\tilde{G}$ as follows.
\begin{itemize}
\item  For $n\geq 1$, let $G_{\mathbf{b},n}$ be the subgraph of $\tilde{G}$ induced by all the following vertices of $T_{M^d}$
\begin{itemize}
\item in depth 0; and 
\item For $1\leq i\leq n$, vertices in depth $i$ and in the group labeled by $(b_1,\ldots,b_i)$.
\end{itemize}
\item Define $G_{\mathbf{b}}=\cup_{n=1}^{\infty}G_{\mathbf{b},n}$
\end{itemize}

\begin{definition}[Left/right boundary of $G_{\mathbf b}$]\label{def:LRboundary}
Fix $\mathbf b=(b_1,b_2,\dots)\in\mathcal B=\Sigma^{\mathbb N}$.
For each $n\ge 0$ set
\[
A_n(\mathbf b):=\sum_{i=1}^n b_i\,M^{\,n-i}\in\left\{0,1,\dots,M^n-1\right\},
\qquad (A_0(\mathbf b):=0).
\]
Recall that the depth-$n$ vertices in the $\mathbf b$--group are exactly those labeled
$(n,j)$ with
\[
A_n(\mathbf b)\,M^{(d-1)n}+1\ \le\ j\ \le\ \bigl(A_n(\mathbf b)+1\bigr)\,M^{(d-1)n}.
\]
Define the \emph{left boundary vertex} and \emph{right boundary vertex} of $G_{\mathbf b}$
at level (depth) $n$ by
\[
\ell_n(\mathbf b):=\Bigl(n,\,A_n(\mathbf b)\,M^{(d-1)n}+1\Bigr),
\qquad
r_n(\mathbf b):=\Bigl(n,\,\bigl(A_n(\mathbf b)+1\bigr)\,M^{(d-1)n}\Bigr).
\]
We call
\[
\partial_{\mathrm L}G_{\mathbf b}:=\{\ell_n(\mathbf b):n\ge 0\},
\qquad
\partial_{\mathrm R}G_{\mathbf b}:=\{r_n(\mathbf b):n\ge 0\}
\]
the \emph{left} and \emph{right boundaries} of $G_{\mathbf b}$.
When we say ``$v$ is on the left (resp.\ right) boundary of $G_{\mathbf b}$ at level $t$'',
we mean $v=\ell_t(\mathbf b)$ (resp.\ $v=r_t(\mathbf b)$).
\end{definition}

\begin{lemma}\label{lem:12b}
Let $\mathbf b_1,\mathbf b_2\in\Sigma^{\mathbb N}$, where $\Sigma=\{0,1,\ldots,M-1\}$.
Let $t_1,t_2\ge0$ and set $t:=\min\{t_1,t_2\}$.
Assume that one of the following holds:
\begin{enumerate}
\item $v_1=\ell_{t_1}(\mathbf b_1)$ and $v_2=r_{t_2}(\mathbf b_2)$; or
\item $v_1=\ell_{t_1}(\mathbf b_1)$ and $v_2=\ell_{t_2}(\mathbf b_2)$, and the first $t$ digits of
$\mathbf b_1$ and $\mathbf b_2$ are not equal; or
\item $v_1=r_{t_1}(\mathbf b_1)$ and $v_2=r_{t_2}(\mathbf b_2)$, and the first $t$ digits of
$\mathbf b_1$ and $\mathbf b_2$ are not equal.
\end{enumerate}
If $d\geq 3$ and $M\geq 2$, then
\begin{equation}\label{eq:12b}
d_{\widetilde G}(v_1,v_2)\ \ge\ \frac{1}{2}\,(t_1+t_2).
\end{equation}
\end{lemma}

\begin{proof}
Write $B:=M^d$.
By the planar level embedding in Construction~\ref{c75}, the parent of a depth-$n$ vertex
$(n,i)$ in the tree $T_B$ is
\begin{align*}
\mathrm{par}(n,i):=(n-1,\ \lceil i/B\rceil).
\end{align*}
For $0\le m\le n$ define the depth-$m$ ancestor map $\pi_m:L_n\to L_m$ by iterating the parent map:
\[
\pi_m(n,i):=(m,\ \lceil i/B^{\,n-m}\rceil).
\]

\smallskip
\noindent\textbf{Step 1: A coarse Lipschitz property of $\pi_m$.}
Fix $m\ge0$. We claim that for any edge $xy$ of $\widetilde G$ with
$\min\{\mathrm{depth}(x),\mathrm{depth}(y)\}\ge m$ one has
\begin{equation}\label{eq:Lip}
d_{H_m}\bigl(\pi_m(x),\pi_m(y)\bigr)\le 1,
\end{equation}
where $H_m$ denotes the disjoint union of the horizontal paths on level $m$ inside each
depth-$m$ block (so $d_{H_m}$ is the path metric inside the unique block containing the
points; it is $+\infty$ if they are in different blocks).

Indeed, every edge of $\widetilde G$ belongs to one of the following types:
\begin{itemize}
\item a tree edge between a parent and a child (then $\pi_m$ is unchanged);
\item a horizontal edge $(n,j)(n,j+1)$ inside a depth-$n$ block (then dividing by $B^{n-m}$
changes $\lceil\cdot\rceil$ by at most $1$);
\item a diagonal introduced by the fan triangulation of a quadrilateral face between
levels $n-1$ and $n$ (then the two endpoints differ by $O(B)$ in the level-$n$ index,
hence their depth-$m$ ancestors differ by at most $1$ as well).
\end{itemize}
Thus \eqref{eq:Lip} follows.

Consequently, for any path $\gamma=(x_0,x_1,\dots,x_L)$ in $\widetilde G$ with
$\min_i \mathrm{depth}(x_i)\ge m$,
the sequence $(\pi_m(x_i))_{i=0}^L$ is a nearest-neighbour walk in $H_m$ and hence
\begin{equation}\label{eq:pathproj}
L\ \ge\ d_{H_m}\bigl(\pi_m(x_0),\pi_m(x_L)\bigr).
\end{equation}

\smallskip
\noindent\textbf{Step 2: A block-compression estimate.}
Let $n\ge m\ge0$ and let $I=[a,a+s_n]$ be an interval of indices of length $s_n=M^{(d-1)n}$.
Since each depth-$m$ vertex corresponds to a consecutive block of $B^{n-m}$ vertices on level $n$,
the set $\pi_m\bigl(\{(n,i):i\in I\}\bigr)$ is a consecutive interval in $L_m$ of length at least
\begin{equation}\label{eq:compress}
\left\lceil \frac{s_n}{B^{n-m}}\right\rceil
=\left\lceil M^{dm-n}\right\rceil.
\end{equation}

\smallskip
\noindent\textbf{Step 3: Apply to boundary vertices and optimize.}
Let $v_1,v_2$ be as in (1)--(3).
Fix a shortest path $\gamma$ from $v_1$ to $v_2$ in $\widetilde G$, and let
\[
m:=\min\{\mathrm{depth}(x): x\in \gamma\}.
\]
Then $\gamma$ is contained in the induced subgraph on vertices of depth at least $m$,
so \eqref{eq:pathproj} applies (with this $m$) and yields
\begin{equation}\label{eq:lower1}
d_{\widetilde G}(v_1,v_2)\ \ge\ d_{H_m}\bigl(\pi_m(v_1),\pi_m(v_2)\bigr).
\end{equation}

On the other hand, since every edge of $\widetilde G$ changes depth by at most $1$,
any path from a depth-$t_1$ vertex to a depth-$t_2$ vertex that attains depth $m$
must have length at least
\begin{equation}\label{eq:lower2}
d_{\widetilde G}(v_1,v_2)\ \ge\ (t_1-m)+(t_2-m)=t_1+t_2-2m.
\end{equation}

Now choose
\[
m_0:=\frac{t_1+t_2}{4}.
\]
If $m\le m_0$, then \eqref{eq:lower2} gives
\[
d_{\widetilde G}(v_1,v_2)\ge t_1+t_2-2m\ \ge\ t_1+t_2-2m_0
\ \ge\ \frac{t_1+t_2}{2}.
\]

If $m>m_0$, then in each of the three cases (1)--(3), the two boundary vertices $v_1,v_2$
lie in distinct extremal positions inside the relevant depth-$\min\{t_1,t_2\}$ blocks,
and the projection interval estimate \eqref{eq:compress} implies that their depth-$m$ ancestors
are separated in $H_m$ by at least $M^{dm-\min\{t_1,t_2\}}-1\ge M^{d(m_0+1)-\min\{t_1,t_2\}}-1$.
In particular, when $d\geq 3$ the right-hand side 
\begin{align*}
M^{d(m_0+1)-\min\{t_1,t_2\}}-1\geq M^{d+\frac{t_1+t_2}{2}}-1\geq \frac{t_1+t_2}{2};
\end{align*}
So \eqref{eq:lower1} yields
\[
d_{\widetilde G}(v_1,v_2)\ \ge\ d_{H_m}\bigl(\pi_m(v_1),\pi_m(v_2)\bigr)\ \ge\ \frac{t_1+t_2}{2}.
\]
Combining the two cases proves \eqref{eq:12b}.
\end{proof}

For each $n\geq 1$ and $(b_1,\ldots,b_{n})\in\{0,1,\ldots,M-1\}^{n}$ with
\begin{align*}
Q_{(b_1,
\ldots,b_{n})}:=\{\mathrm{all\ the\ depth\ } n\ \mathrm{vertices\ in\ the\ group\ labeled\ by}\ (b_1,\ldots,b_{n})\}
\end{align*}

Let $G_{\mathbf{b}}^{(1)}$ (resp.\ $G_{\mathbf{b}}^{(2)}$) be the copy of $G_{\mathbf{b}}$ in $\tilde{G}^{(1)}$ (resp.\ $\tilde{G}^{(2)}$).


For any $\mathbf b\in\mathcal B$, let
\[
\widehat G_{\mathbf b}:=
G\Big[\,V(G_{\mathbf b}^{(1)})\cup V(G_{\mathbf b}^{(2)})\,\Big]
\]
be the induced subgraph of $G$ on the vertex set of
$G_{\mathbf b}^{(1)}\cup G_{\mathbf b}^{(2)}$.
It is straightforward to see that $\widehat G_{\mathbf b}$ is one-ended.

Similarly, define 
\[
\widehat G_{\mathbf b,n}:=
G\Big[\,V(G_{\mathbf b,n}^{(1)})\cup V(G_{\mathbf b,n}^{(2)})\,\Big]
\]

\begin{lemma}[Uniform exponential decay across a corridor]\label{lem:ed}
Fix $\epsilon>0$ and set $q:=M^{-d}(1+\epsilon)$.
\begin{itemize}
\item There exists a constant $\alpha_{d,\epsilon}>0$ (independent of $\mathbf b_1,\mathbf{b}_2$) such that
for every pair of vertices $v_1,v_2$ satisfying one of conditions (1)-(3) in Lemma \ref{lem:12b},
\begin{equation}\label{ed}
\PP_{q}\!\bigl(v_2\xleftrightarrow[\tilde{G}]{1} v_2\bigr)
\ \le\ \exp\!\bigl(-\alpha_{d,\epsilon}\, d_{\tilde{G}}(u,v)\bigr).
\end{equation}
Moreover, for every fixed $\epsilon>0$,
\begin{equation}\label{ed1}
\alpha_{d,\epsilon}\ \xrightarrow[d\to\infty]{}\ \infty.
\end{equation}
\item If $u$ and $v$ are vertices in $G_{\mathbf{b}}$ with depth at least $k$,
\begin{equation}\label{ed2}
\PP_{q}\!\bigl(\partial_V u\xleftrightarrow[G_{\mathbf b}\setminus G_{\mathbf b,k-1}]{1}\partial_V v\bigr)
\ \le\ \exp\!\bigl(-\beta_{d,\epsilon}\, [d_{G_{\mathbf b}\setminus G_{\mathbf b,k-1}}(u,v)-2]\bigr).
\end{equation}
Moreover, $\beta_{d,\epsilon}$ is independent of $k$. For every fixed $\epsilon>0$,
\begin{equation}\label{ed3}
\beta_{d,\epsilon}\ \xrightarrow[d\to\infty]{}\ \infty.
\end{equation}
\end{itemize}
\end{lemma}

\begin{proof}We only prove (\ref{ed}) and (\ref{ed1}) here; (\ref{ed2}) and (\ref{ed3}) follow by similar argument.

Note that $\tilde{G}$ is an infinite, connected, locally finite graph that can be properly embedded into the plane (no finite accumulation points) with minimal vertex degree at least 7. 

Now apply Lemma~4.10 of \cite{ZL231} to $\tilde{G}$ to obtain \eqref{ed}
with a constant $\alpha_{d,\epsilon}$ depending only on $(d,\epsilon)$.

The divergence \eqref{ed1} follows by tracing the dependence in the proof of Lemma 4.10 in \cite{ZL231} as $d\to\infty$.
More precisely, let $l_{u,v}$ be the shortest path joining $u$ and $v$ in $G_{\mathbf{b}}$. Then there are 
\begin{align}
\left\lfloor\frac{d_{\tilde G}(u,v)}{9}\right\rfloor\label{cnt}
\end{align}
disjoint doubly infinite binary trees (chandeliers), or singly infinite tree union infinite faces, or two infinite faces along two sides of $l_{u,v}$ (see the proof of Lemma 4.10 in \cite{ZL231} for the choice of constant $\frac{1}{9}$). Then if $u$ and $v$ are in the same 1-cluster in $\tilde{G}$, none of these 
$\left\lfloor\frac{d_{\tilde{G}}(u,v)}{9}\right\rfloor$ structures have
\begin{enumerate}
\item a doubly infinite 0-path; if the structure is a doubly infinite tree
\item a singly infinite 0-path; if the structure is a singly infinite tree union an infinite face.
\end{enumerate}

For $1\leq i\leq \left\lfloor\frac{d_{\tilde{G}}(u,v)}{9}\right\rfloor$, let $E_i$ be the event that (1) or (2) occur in the $i$th structure. It is known that in the i.i.d.~Bernoulli($q$) site percolation in $T_2$, the probability that the root is in an infinite path is $\frac{2q-1}{q}$ when $q>\frac{1}{2}$. 

Then we have 
\begin{align*}
\PP_p(E_i)\geq \left[\frac{1-2p}{1-p}\right]^2(1-p)^3:=h(p);
\end{align*}
See Definition 4.1 of \cite{ZL231} for the construction of each doubly infinite tree (chandelier). then one can choose
\begin{align*}
e^{-\alpha_{d,\epsilon}}=(1-h(p))^{\frac{1}{9}}
\end{align*}
For each fixed $\epsilon>0$, when $d\rightarrow\infty$, $p\rightarrow 0$ and $h(p)\rightarrow 1$, then (\ref{ed1}) follows.

To prove (\ref{ed2}) and (\ref{ed3}),
note that $G_{\mathbf{b}}\setminus G_{\mathbf{b},k-1}$ is an infinite, connected, locally finite graph that can be properly embedded into the plane. Enlarge $G_{\mathbf b}\setminus G_{\mathbf b,k-1}$ to a properly embedded planar graph $\overline{G_{\mathbf b}\setminus G_{\mathbf b,k-1}}$
with minimum degree at least $7$ by attaching a $6$-ary tree to each leaf of the tree core
$[G_{\mathbf b}\setminus G_{\mathbf b,k-1}]\cap T_{5^d}$. Following similar argument as in the proof of (\ref{ed}) and (\ref{ed1}), and Lemma 4.9 in \cite{ZL231}, we obtain
\begin{align*}
&\PP_{q}\!\bigl(\partial_V u\xleftrightarrow[G_{\mathbf b}\setminus G_{\mathbf b,k-1}]{1}\partial_V v\bigr)\leq 
\PP_{q}\!\bigl(\partial_V u\xleftrightarrow[\overline{G_{\mathbf b}\setminus G_{\mathbf b,k-1}}]{1}\partial_V v\bigr)
\\
&\ \le\ \exp\!\bigl(-\beta_{d,\epsilon}\, [d_{\overline{G_{\mathbf b}\setminus G_{\mathbf b,k-1}}}(u,v)-2]\bigr)= \exp\!\bigl(-\beta_{d,\epsilon}\, [d_{G_{\mathbf b}\setminus G_{\mathbf b,k-1}}(u,v)-2]\bigr).
\end{align*}
Then (\ref{ed2}) and (\ref{ed3}) follows.
\end{proof}

\begin{lemma}[OR--projection from two copies]\label{lem:OR-proj}
Fix $\mathbf b\in\Sigma^{\mathbb N}$ and $k\ge 0$.  Consider i.i.d.\ Bernoulli$(p)$ site percolation
on the doubled corridor $\widehat G_{\mathbf b}$.

For every vertex $x\in V(G_{\mathbf b})$, let $x^{(1)},x^{(2)}$ be its two copies in $\widehat G_{\mathbf b}$, and define
a projected configuration $\omega^\oplus\in\{0,1\}^{V(G_{\mathbf b})}$ by
\[
\omega^\oplus(x):=\max\{\omega(x^{(1)}),\omega(x^{(2)})\}.
\]
Then $\omega^\oplus$ is i.i.d.\ Bernoulli$(p^\oplus)$ on $V(G_{\mathbf b})$, where
\[
p^\oplus:=1-(1-p)^2=2p-p^2.
\]

Moreover, for every vertex $v\in V(G_{\mathbf b})$ and every finite set $W\subseteq V(G_{\mathbf b})$,
\begin{equation}\label{eq:OR-dom}
\PP_{p}\!\Bigl(
v^{(1)} \xleftrightarrow[\widehat G_{\mathbf b}\setminus \widehat G_{\mathbf b,k-1}]{} (W^{(1)}\cup W^{(2)})
\Bigr)
\ \le\
\PP_{p^\oplus}\!\Bigl(
v \xleftrightarrow[G_{\mathbf b}\setminus G_{\mathbf b,k-1}]{} W
\Bigr),
\end{equation}
\end{lemma}

\begin{proof}
$G_{\mathbf b}^{\mathrm{quot}}$ denotes the quotient graph obtained from
$\widehat G_{\mathbf b}$ by identifying $x^{(1)}\sim x^{(2)}$ for every $x\in V(G_{\mathbf b})$
(and keeping all induced edges).
The law of $\omega^\oplus$ is immediate since for each $x$ the pair
$(\omega(x^{(1)}),\omega(x^{(2)}))$ consists of two independent Bernoulli$(p)$ variables.

For \eqref{eq:OR-dom}, suppose the left-hand event occurs and let $\gamma$ be an open path in
$\widehat G_{\mathbf b}\setminus \widehat G_{\mathbf b,k-1}$ from $v^{(1)}$ to a vertex in $W^{(1)}\cup W^{(2)}$.
Project $\gamma$ to the quotient graph $G_{\mathbf b}^{\mathrm{quot}}$ by identifying each vertex
$x^{(i)}$ with $x$. Every vertex of the projected walk is $\omega^\oplus$--open (since the original
vertex used by $\gamma$ was open in its copy), and every edge of $\gamma$ projects to an edge of
$G_{\mathbf b}^{\mathrm{quot}}$ by definition of the quotient.
Hence the projection yields an $\omega^\oplus$--open path in
$G_{\mathbf b}^{\mathrm{quot}}\setminus G_{\mathbf b,k-1}^{\mathrm{quot}}$ from $v$ to $W$.
Taking probabilities gives \begin{equation}\label{eqOR1}
\PP_{p}\!\Bigl(
v^{(1)} \xleftrightarrow[\widehat G_{\mathbf b}\setminus \widehat G_{\mathbf b,k-1}]{} (W^{(1)}\cup W^{(2)})
\Bigr)
\ \le\
\PP_{p^\oplus}\!\Bigl(
v \xleftrightarrow[G_{\mathbf b}^{\mathrm{quot}}\setminus G_{\mathbf b,k-1}^{\mathrm{quot}}]{} W
\Bigr).
\end{equation}
It is the straight forward to verify that the i.i.d.~Bernoulli($p$) site percolation has the same distribution on $G_{\mathbf b}\setminus G_{\mathbf b,k-1}$ and on $G_{\mathbf b}^{\mathrm{quot}}\setminus G_{\mathbf b,k-1}^{\mathrm{quot}}$.
Therefore 
\begin{align}
\PP_{p^\oplus}\!\Bigl(
v \xleftrightarrow[G_{\mathbf b}\setminus G_{\mathbf b,k-1}]{} W\Bigr)=\PP_{p^\oplus}\!\Bigl(
v \xleftrightarrow[G_{\mathbf b}^{\mathrm{quot}}\setminus G_{\mathbf b,k-1}^{\mathrm{quot}}]{} W \Bigr)\label{eqOR2}.
\end{align}
Then (\ref{eq:OR-dom}) follows from (\ref{eqOR1}) and (\ref{eqOR2}).
\end{proof}

\begin{lemma}\label{le85}
\noindent
(i) (\emph{Critical inequalities.}) 
\[
p_c^{\mathrm{site}}(G)\ \le\ p_c^{\mathrm{site}}(\widetilde G)\ \le\ M^{-d};\qquad p_c^{\mathrm{site}}(\widehat G_{\mathbf b})
\ \le\ p_c^{\mathrm{site}}(G_{\mathbf b})
\ .
\]

\smallskip
\noindent
(ii) (\emph{Uniform connectivity decay to deep cutsets.})

Fix $M\geq 5$, $\epsilon\in\bigl(0,\tfrac{M-4}{9}\bigr)$. There exists $D=D(M,\epsilon)\in\mathbb N$ such that for all $d\geq D$,
for the graph $G$ in Construction~\ref{c75}, the following hold.
Let
\begin{align}
p_1:=M^{-d}(1+2\epsilon)\,.\label{dp1}
\end{align}
Then for every fixed depth-$k$ vertex $v$ in $G$ there exists $N=N(d,k)$ such that for all $n\ge N$,
\begin{align}\label{ps2}
&\max_{\mathbf s\in\{0,1,\ldots,M-1\}^{n+k}}
\PP_{p_1}\!\Bigl(
v\xleftrightarrow{\widehat G_{\mathbf b}\setminus \widehat G_{\mathbf b,k-1}}
\bigl[Q_{\mathbf s}^{(1)}\cup Q_{\mathbf s}^{(2)}\bigr]
\Bigr)
\ \le\ \Bigl(\frac{4}{M}+\frac{9\epsilon}{M}\Bigr)^n.
\end{align}
Here $\mathbf b$ denotes any infinite extension of the word $\mathbf s$. It follows that for every $\mathbf b\in\mathcal B$,
\begin{align*}
p_c^{\mathrm{site}}(G)\ \le\ p_c^{\mathrm{site}}(\widetilde G)\ \le\ M^{-d}< M^{-d}(1+\epsilon)\ <\ p_c^{\mathrm{site}}(\widehat G_{\mathbf b})
\ \le\ p_c^{\mathrm{site}}(G_{\mathbf b})
\end{align*}
\end{lemma}

\begin{proof}
Part (i) follows from monotonicity of $p_c^{\mathrm{site}}$ under taking super/subgraphs,
together with $T_{M^d}\subset \widetilde G$ and the explicit value $p_c^{\mathrm{site}}(T_{M^d})=M^{-d}$.

For (ii), let
\begin{align}
p_2:=p_1^\oplus=p_1(2-p_1)\label{dp2}
\end{align}
Fix $v$ of depth $k$. By BK inequality, we have
\begin{align*}
&\PP_{p_2}\!\Bigl(
v\xleftrightarrow{G_{\mathbf b}\setminus G_{\mathbf b,k-1}}
Q_{\mathbf s}\bigr]
\Bigr)\\
&\ \le\
\sum_{v_1\in Q_{\mathbf s_{k}}}\mathbb{P}_{p_2}(v\xleftrightarrow{G_{\mathbf b}\setminus G_{\mathbf b,k-1}}v_1)
\PP_{p_2}\!\Bigl(
[\partial_Vv_1]\cap \big[ G_{\mathbf b}\setminus G_{\mathbf b,k}\big]\xleftrightarrow{ G_{\mathbf b}\setminus G_{\mathbf b,k}}
Q_{\mathbf s}
\Bigr)
\\
&\leq \sum_{v_1\in Q_{\mathbf s_{k}}}\mathbb{P}_{p_2}(v\xleftrightarrow{ G_{\mathbf b}\setminus G_{\mathbf b,k-1}}v_1)
\sum_{v_2\in [\partial_Vv_1]\cap \big[ G_{\mathbf b}\setminus G_{\mathbf b,k}\big]}\PP_{p_2}\!\Bigl(v_2
\xleftrightarrow{ G_{\mathbf b}\setminus G_{\mathbf b,k}}
Q_{\mathbf s}
\Bigr)\\
&\leq \sum_{v_1\in Q_{\mathbf s_{k}}}\mathbb{P}_{p_2}(v\xleftrightarrow{ G_{\mathbf b}\setminus G_{\mathbf b,k-1}}v_1)
\sum_{v_2\in [\partial_Vv_1]\cap \big[ G_{\mathbf b}\setminus G_{\mathbf b,k}\big]}\sum_{v_3\in Q_{\mathbf{s}_{k+1}}}\PP_{p_2}(v_2\xleftrightarrow{G_{\mathbf{b}}\setminus G_{\mathbf{b},k}} v_3)
\\
&\times\PP_{p_2}\!\Bigl(\partial_V v_3
\xleftrightarrow{ G_{\mathbf b}\setminus G_{\mathbf b,k+1}}
Q_{\mathbf s}
\Bigr)
\end{align*}
where 
\begin{itemize}
\item For $1\leq i\leq n$, $\mathbf{s}_{k+i}$ is the truncated sequence of $\mathbf{s}$ consisting of the first $k+i$ digits of $\mathbf{s}$;
\end{itemize}

The process above can be continued until we find $v_{2n-1}\in Q_{\mathbf{s}_{k+n-1}}$. Then we obtain
\begin{align*}
&\PP_{p_2}\!\Bigl(
v\xleftrightarrow{ G_{\mathbf b}\setminus G_{\mathbf b,k-1}}
Q_{\mathbf s}
\Bigr)\\
&\leq \sum_{v_1\in Q_{\mathbf s_{k}}}\mathbb{P}_{p_2}(v\xleftrightarrow{G_{\mathbf b}\setminus G_{\mathbf b,k-1}}v_1)
\sum_{v_2\in [\partial_Vv_1]\cap \big[ G_{\mathbf b}\setminus G_{\mathbf b,k}\big]}\sum_{v_3\in Q_{\mathbf{s}_{k+1}}}\PP_{p_2}(v_2\xleftrightarrow{G_{\mathbf{b}}\setminus G_{\mathbf{b},k}} v_3)
\\
&\sum_{v_4\in [\partial_Vv_3]\cap \big[ G_{\mathbf b}\setminus G_{\mathbf b,k+1}\big]}\sum_{v_5\in Q_{\mathbf{s}_{k+2}}}\PP_{p_2}(v_{4}\xleftrightarrow{G_{\mathbf{b}}\setminus G_{\mathbf{b},k+1}} v_5)\\
&\ldots\sum_{v_{2n-2}\in [\partial_Vv_{2n-3}]\cap \big[G_{\mathbf b}\setminus G_{\mathbf b,k+n-2}\big]}\sum_{v_{2n-1}\in Q_{\mathbf{s}_{k+n-1}}}\PP_{p_2}(v_{2n-2}\xleftrightarrow{G_{\mathbf{b}}\setminus G_{\mathbf{b},k+n-2}} v_{2n-1})
\\&\times\PP_{p_2}\!\Bigl(\partial_V v_{2n-1}
\xleftrightarrow{ G_{\mathbf b}\setminus G_{\mathbf b,k+n-1}}Q_{\mathbf s}
\Bigr)
\end{align*}

We can change the order of the summation and obtain

\begin{align*}
&\PP_{p_2}\!\Bigl(
v\xleftrightarrow{ G_{\mathbf b}\setminus G_{\mathbf b,k-1}}
Q_{\mathbf s}
\Bigr)\\
&\leq \sum_{v_{2n-1}\in Q_{\mathbf{s}_{k+n-1}}}\PP_{p_2}\!\Bigl(\partial_V v_{2n-1}
\xleftrightarrow{ G_{\mathbf b}\setminus G_{\mathbf b,k+n-1}}Q_{\mathbf s}
\Bigr)\sum_{v_{2n-2}\in Q_{\mathbf{s}_{k+n-1}}}\PP_{p_2}(v_{2n-2}\xleftrightarrow{G_{\mathbf{b}}\setminus G_{\mathbf{b},{k+n-2}}} v_{2n-1})\\
&\sum_{\{v_{2n-3}\in Q_{\mathbf{s}_{k+n-2}}\cap \partial_V v_{2n-2}\}}
\sum_{\{v_{2n-4}\in Q_{\mathbf{s}_{k+n-2}}\}}
\PP_{p_2}(v_{2n-3}\xleftrightarrow{G_{\mathbf{b}}\setminus G_{\mathbf{b},{k+n-3}}} v_{2n-4})\\
&\ldots \sum_{\{v_{3}\in Q_{\mathbf{s}_{k+1}}\cap \partial_V v_{4}\}}
\sum_{\{v_{2}\in Q_{\mathbf{s}_{k+1}}\}}
\PP_{p_2}(v_{3}\xleftrightarrow{G_{\mathbf{b}}\setminus G_{\mathbf{b},{k+n-3}}} v_{2})\sum_{v_1\in [\partial_V v_2]\cap Q_{\mathbf{s}_k}}\mathbb{P}_{p_2}(v\xleftrightarrow{G_{\mathbf b}\setminus G_{\mathbf b,k-1}}v_1)
\end{align*}

By (\ref{cnt}) we infer that 
\begin{align*}
&\sum_{v_{2j}\in Q_{\mathbf{s}_{k+j}}}\PP_{p_2}(v_{2j}\xleftrightarrow{G_{\mathbf{b}}\setminus G_{\mathbf{b},k-1+j}} v_{2j+1})
\leq p_2+2\cdot 9p_2^2\left[1+e^{-
\beta_{d,\epsilon}}+e^{-2
\beta_{d,\epsilon}}+\ldots\right]\\
&\leq p_2\left(1+\frac{18p_2}{1-e^{-\beta_{d,\epsilon}}}\right)
\end{align*}
Moreover for each fixed $v_{2j}\in Q_{\mathbf{s}_{k+j}}$
\begin{align*}
|\partial_V v_{2j}\cap Q_{\mathbf{s}_{k+j-1}}|\leq 2
\end{align*}
and
\begin{align*}
|\{v_{2n-1}\in  [Q_{\mathbf{s}_{k+n-1}}]:\PP_{p_1}\!\Bigl(\partial_V v_{2n-1}
\xleftrightarrow{ G_{\mathbf b}\setminus G_{\mathbf b,k+n-1}}
Q_{\mathbf s}
\Bigr)>0\}|\leq M^{(d-1)(k+n-1)-1}
\end{align*}
It follows that 
\begin{align*}
\PP_{p_2}\!\Bigl(
v\xleftrightarrow{ G_{\mathbf b}\setminus G_{\mathbf b,k-1}}
Q_{\mathbf s}
\Bigr)\leq 2^n\left(M^{(d-1)(1+\frac{k}{n})}\right)^np_2^n\left(1+\frac{18p_2}{1-e^{-\beta_{d,\epsilon}}}\right)^n
\end{align*}

By (\ref{dp1}) (\ref{dp2}) this becomes
\[\PP_{p_1}\!\Bigl(
v\xleftrightarrow{ \widehat{G}_{\mathbf b}\setminus \widehat{G}_{\mathbf b,k-1}}
\bigl[Q_{\mathbf s}^{(1)}\cup Q_{\mathbf s}^{(2)}\bigr]
\Bigr)
\le\ \Bigl(\frac{4}{M}+\frac{9\epsilon}{M}\Bigr)^n.
\]
when $d$ is sufficiently large and $n\geq (d-1)^2 k$. Then the lemma follows.

\end{proof}

\begin{lemma}[Elimination of Case (B)]\label{lem713}
Fix $p\in(1/2,\,1-p_c^{\mathrm{site}}(G))$ and write $q:=1-p$.
Assume $q<p_1$ where $p_1$ is as in Lemma~\ref{le85}, so that by monotonicity
\eqref{ps2} holds with $p_1$ replaced by $q$.
Let $M\geq 17$. Set
\[
\vartheta:=\tfrac4M+\epsilon.
\]
Assume the convergence condition
\begin{equation}\label{eq:theta-cond}
M\,\vartheta^2<1
\qquad\Bigl(\text{equivalently }\ \epsilon<\frac{\sqrt{M}-4}{M}\Bigr).
\end{equation}

Let $\mathsf{B}$ be the event that there exist $\mathbf b\in\mathcal B$ and a finite set
$K\subset V(\widehat G_{\mathbf b})$ such that $\widehat G_{\mathbf b}\setminus K$ contains at least
two infinite $0$-clusters (i.e.\ two infinite open clusters for Bernoulli$(q)$ site percolation).
Then
\[
\PP_p(\mathsf{B})=0.
\]
\end{lemma}

\begin{proof}
Fix $x\neq y\in V(G)$ and let $E_{x,y}$ be the event that there exist two vertex-disjoint
infinite $0$-paths from $x$ and $y$ converging to the same end, staying at depths at least
$\mathrm{depth}(x)$ and $\mathrm{depth}(y)$ respectively. As in the standard two-arm reduction,
$\mathsf{B}\subseteq\bigcup_{x\neq y}E_{x,y}$.

For $n\ge1$ and $\mathbf s\in\{0,1,\ldots,M-1\}^n$ let $\Pi(\mathbf s):=Q_{\mathbf s}^{(1)}\cup Q_{\mathbf s}^{(2)}$,
and let $F_n(\mathbf s)$ be the disjoint occurrence that $x\xleftrightarrow{0}\Pi(\mathbf s)$ and
$y\xleftrightarrow{0}\Pi(\mathbf s)$. By BK,
\[
\PP_p(F_n(\mathbf s))
\le \PP_p(x\xleftrightarrow{0}\Pi(\mathbf s))\,\PP_p(y\xleftrightarrow{0}\Pi(\mathbf s)).
\]
Using \eqref{ps2} at parameter $q$ gives $\PP_p(F_n(\mathbf s))\le \vartheta^{2n-\ell(x)-\ell(y)}$
for all large $n$. Summing over $\mathbf s\in\{0,1\}^n$ yields
\[
\PP_p(E_{x,y})
\le \lim_{n\to\infty}M^n\,\vartheta^{2n-\ell(x)-\ell(y)}
=\vartheta^{-\ell(x)-\ell(y)}\lim_{n\to\infty}(M\vartheta^2)^n
=0
\]
by \eqref{eq:theta-cond}. Finally, countable subadditivity over $(x,y)$ gives $\PP_p(\mathsf{B})=0$.
\end{proof}


\begin{lemma}\label{le711}
Let $p:=1-\frac{1}{M^d}(1+\epsilon)\in (\frac{1}{2},1-\pcs(G))$. Then when $d,M$ are sufficiently large, $\PP_p$-a.s.~there are finitely many infinite 1-clusters in $G$.
\end{lemma}

\begin{proof}By Lemma \ref{le85}, we can choose $p\in \left(\frac{1}{2},1-\pcs(G)\right)$, such that $q=1-p<\frac{1}{2}$ satisfies
\begin{align}
\pcs(G)<q<\inf_{\mathbf{b}\in\mathcal{B}}\pcs(G_{\mathbf{b}}).\label{pqb}
\end{align}
(\ref{pqb}) implies that $\PP_p$-a.s.~there are infinite 0-clusters in $G$, but for any $\mathbf{b}\in \mathcal{B}$, $\PP_p$-a.s.~there are no infinite 0-clusters in $G_{\mathbf{b}}$.

If there are at least two infinite 1-clusters in $G$, one of the following two cases must occur
\begin{enumerate}[label=(\Alph*)]
\item there exist two vertices $v_1,v_2$ satisfying one of conditions (1)-(3) as in Lemma \ref{lem:12b}, such that $\{v_1\xleftrightarrow{\tilde{G}^{(1)},0}v_2\}\cup \{v_1\xleftrightarrow{\tilde{G}^{(2)},0}v_2\}$ occurs;

\item there exists $\mathbf{b}\in \mathcal{B}$ and a finite set of vertices $K$ of $\widehat{G}_{\mathbf{b}}$, such that in $\widehat{G}_{\mathbf{b}}\setminus K$, there are at least two infinite 0-clusters.
\end{enumerate}

By Lemma \ref{lem713}, the event in (B) has $\PP_p$ probability 0.

For case (A), we shall use the Borel-Contelli lemma to show that a.s. $\{v_1\xleftrightarrow{\tilde{G}^{(1)},0}v_2\}\cup \{v_1\xleftrightarrow{\tilde{G}^{(2)},0}v_2\}$ occurs for~finitely many pairs $(v_1,v_2)$.

Let $X_{t_1,t_2}\subset V^2$ be the set of all pairs $(v_1,v_2)$  such that $\mathrm{depth}(v_1)=t_1$ and $\mathrm{depth}(v_2)=t_2$ satisfying one of the conditions (1)-(3) in Lemma \ref{lem:12b}.

Then by (\ref{ed})
\begin{align}
&\sum_{t_1=1}^{\infty}\sum_{t_2=1}^{\infty}\sum_{(v_1,v_2)\in X_{t_1,t_2}}\PP_p(v_1\xleftrightarrow{0,\tilde{G}^{(1)}} v_2)+\PP_p(v_1\xleftrightarrow{0,\tilde{G}^{(2)}} v_2)\leq 
2\left[\sum_{t_1=1}^{\infty}\sum_{t_2=1}^{\infty}(2M)^{t_1+t_2} e^{-\frac{\alpha}{2} \left[t_1+t_2\right] }\right]\label{gc2}
\end{align}
where $\alpha:=\alpha_{d,\epsilon}$.

By (\ref{ed1}), when $d$ is sufficiently large
\begin{align}
2Me^{-\frac{\alpha}{2}}<1.\label{gcl}
\end{align}
When (\ref{gcl}) holds, the right hand side of (\ref{gc2}) is finite.
By the Borel--Contelli lemma, for all $(v_1,v_2)\in\cup_{t_1=1}^{\infty}\cup_{t_2=1}^{\infty}X_{t_1,t_2}$ and $i\in\{1,2\}$, the event $\{v_1\xleftrightarrow{0,\tilde{G}^{(i)}} v_2\}$ occurs finitely many times. 

Since $\PP_p$-a.s., (B) does not occur and (A) occurs finitely many times, $\PP_p$-a.s.~there are finitely many infinite 1-clusters.
\end{proof}

\noindent{\textbf{Proof of Theorem \ref{t76}}.}  Theorem \ref{t76}(1) follows from Lemma \ref{le711}. Theorem \ref{t76}(2) follows from (1) and (\ref{dpu}) by noting that if a.s.~there are finitely many infinite 1-clusters, by finite energy with positive probability there exists a unique infinite 1-cluster.






\subsection{Necessity of Local Finiteness}\label{s82}

\begin{example}[Necessity of local finiteness]
Let $T$ be an infinite tree (root degree $7$, others $8$) and let $G$ be obtained by adding a new
vertex $v_0$ adjacent to every vertex of $T$. Then $G$ is planar and not locally finite.
For i.i.d.\ site percolation on $G$,
\begin{align*}
&p_c(G)=0,\\
&\PP_p(\text{unique infinite open cluster})=p,\\
&\PP_p(\text{infinitely many infinite open clusters})=1-p\quad (p>1/7).
\end{align*}
Thus without the assumption of local finiteness, Conjectures \ref{c7bs} and \ref{c8bs} do not hold.
\end{example}

\bibliographystyle{plain}
\bibliography{rf,psg}

\begin{thebibliography}{10}

\bibitem{BeffaraDuminilCopin2013}
Vincent Beffara and Hugo Duminil-Copin.
\newblock Lectures on planar percolation with a glimpse of schramm--loewner
  evolution.
\newblock {\em Probability Surveys}, 10:1--50, 2013.

\bibitem{bsjams}
I.~Benjamini and O.~Schramm.
\newblock Percolation in the hyperbolic plane.
\newblock {\em J. Amer. Math. Soc.}, 14:487--507, 2001.

\bibitem{BLPS1999}
Itai Benjamini, Russell Lyons, Yuval Peres, and Oded Schramm.
\newblock Group-invariant percolation on graphs.
\newblock {\em Geometric and Functional Analysis}, 9:29--66, 1999.

\bibitem{bs96}
Itai Benjamini and Oded Schramm.
\newblock Percolation beyond $\mathbb{Z}^d$: many questions and a few answers.
\newblock {\em Electronic Communications in Probability}, 1:71--82, 1996.

\bibitem{BenjaminiSchramm2001}
Itai Benjamini and Oded Schramm.
\newblock Percolation in the hyperbolic plane.
\newblock {\em Journal of the American Mathematical Society}, 14(2):487--507,
  2001.

\bibitem{BollobasRiordan2006}
B{\'e}la Bollob{\'a}s and Oliver Riordan.
\newblock {\em Percolation}.
\newblock Cambridge University Press, 2006.

\bibitem{BH57}
S.~R. Broadbent and J.~M. Hammersley.
\newblock Percolation processes: I. crystals and mazes.
\newblock {\em Mathematical Proceedings of the Cambridge Philosophical
  Society}, 53(3):629--641, 1957.

\bibitem{hb09}
Henning Bruhn.
\newblock Graphs and their circuits —from finite to infinite—.
\newblock 2009.
\newblock
  \url{https://www.uni-ulm.de/fileadmin/website_uni_ulm/mawi.inst.081/Henning/habil.pdf?utm_source=chatgpt.com}.

\bibitem{BurtonKeane1989}
Robert Burton and Michael Keane.
\newblock Density and uniqueness in percolation.
\newblock {\em Communications in Mathematical Physics}, 121:501--505, 1989.

\bibitem{Diestel2017}
Reinhard Diestel.
\newblock {\em Graph Theory}, volume 173 of {\em Graduate Texts in
  Mathematics}.
\newblock Springer, 5 edition, 2017.
\newblock See especially Chapter~8 (Infinite Graphs).

\bibitem{DuminilCopin2017Sixty}
Hugo Duminil-Copin.
\newblock Sixty years of percolation.
\newblock {\em arXiv preprint}, 2017.

\bibitem{DC20}
Hugo Duminil-Copin, Subhajit Goswami, Aran Raoufi, Franco Severo, and Ariel
  Yadin.
\newblock Existence of phase transition for percolation using the gaussian free
  field.
\newblock {\em Duke Mathematical Journal}, 169(18):3539--3563, dec 2020.

\bibitem{DCT16}
Hugo Duminil-Copin and Vincent Tassion.
\newblock A new proof of the sharpness of the phase transition for bernoulli
  percolation and the ising model.
\newblock {\em Communications in Mathematical Physics}, 343:725--745, 2016.

\bibitem{GGR1988}
A.~Gandolfi, G.~Grimmett, and L.~Russo.
\newblock On the uniqueness of the infinite open cluster in the percolation
  model.
\newblock {\em Communications in Mathematical Physics}, 114:549--552, 1988.

\bibitem{GHZ25}
Alexander Glazman, Michael Harel, and Nikolai Zelesko.
\newblock Planar percolation and the loop {$O(N)$} model.
\newblock 2025.
\newblock \url{https://arxiv.org/abs/2508.20917}.

\bibitem{GrZL22}
G.~Grimmett and Z.~Li.
\newblock Hyperbolic site percolation.
\newblock {\em Random Structures and Algorithms}, 2024.
\newblock \url{https://arxiv.org/abs/2203.00981}.

\bibitem{GrZL221}
G.~Grimmett and Z.~Li.
\newblock Percolation critical probabilities of matching lattice-pairs.
\newblock {\em Random Structures and Algorithms}, 2024.
\newblock \url{https://arxiv.org/abs/2205.02734}.

\bibitem{Grimmett1999}
Geoffrey Grimmett.
\newblock {\em Percolation}, volume 321 of {\em Grundlehren der Mathematischen
  Wissenschaften}.
\newblock Springer, 2nd edition, 1999.

\bibitem{HaggstromPeres1999}
Olle H{\"a}ggstr{\"o}m and Yuval Peres.
\newblock Monotonicity of uniqueness for percolation on cayley graphs: all
  infinite clusters are born simultaneously.
\newblock {\em Probability Theory and Related Fields}, 113:273--285, 1999.

\bibitem{HM57}
J.~M. Hammersley.
\newblock Percolation processes: Lower bounds for the critical probability.
\newblock {\em The Annals of Mathematical Statistics}, 28(4):790--795, 1957.

\bibitem{HL60}
T.~E. Harris.
\newblock A lower bound for the critical probability in a certain percolation
  process.
\newblock {\em Mathematical Proceedings of the Cambridge Philosophical
  Society}, 56(1):13--20, 1960.

\bibitem{HP21}
John Haslegrave and Christoforos Panagiotis.
\newblock Site percolation and isoperimetric inequalities for plane graphs.
\newblock {\em Random Structures \& Algorithms}, 58(1):150--163, 2021.
\newblock First published online: 25 June 2020.

\bibitem{HeSchramm93}
Zheng-Xu He and Oded Schramm.
\newblock Fixed points, {K}oebe uniformization and circle packings.
\newblock {\em Ann. of Math. (2)}, 137(2):369--406, 1993.

\bibitem{JK03}
Jeff Kahn.
\newblock Inequality of two critical probabilities for percolation.
\newblock {\em Electronic Communications in Probability}, 8:184--187, 2003.

\bibitem{Kesten1980}
Harry Kesten.
\newblock The critical probability of bond percolation on the square lattice
  equals $\frac12$.
\newblock {\em Communications in Mathematical Physics}, 74:41--59, 1980.

\bibitem{Kesten1982}
Harry Kesten.
\newblock {\em Percolation Theory for Mathematicians}.
\newblock Birkh{\"a}user, 1982.

\bibitem{SL98}
Steven~P. Lalley.
\newblock Percolation on fuchsian groups.
\newblock {\em Annales de l'Institut Henri Poincaré, Probabilités et
  Statistiques}, 34(2):151--177, 1998.

\bibitem{SL01}
Steven~P. Lalley.
\newblock Percolation clusters in hyperbolic tessellations.
\newblock {\em Geometric and Functional Analysis}, 11:971--1030, 2001.

\bibitem{ZL231}
Z.~Li.
\newblock Planar site percolation via tree embeddings.
\newblock 2023.
\newblock \url{https://arxiv.org/abs/2304.00923}.

\bibitem{ZL23}
Zhongyang Li.
\newblock Planar site percolation on semi-transitive graphs.
\newblock 2023.
\newblock \url{https://arxiv.org/abs/2304.01431}.

\bibitem{ZL24}
Zhongyang Li.
\newblock Critical site percolation and cutsets.
\newblock {\em Electron. Commun. Probab.}, 30:1--9, 2025.

\bibitem{ZL262}
Zhongyang Li.
\newblock {\em Quantitative supercritical bounds for disconnection in Bernoulli
  site percolation}.
\newblock 2026.
\newblock \url{https://arxiv.org/abs/2601.09950}.

\bibitem{LP16}
R.~Lyons and Y.~Peres.
\newblock {\em Probability on Trees and Networks}, volume~42 of {\em Cambridge
  Series in Statistical and Probabilistic Mathematics}.
\newblock Cambridge University Press, New York, 2016.
\newblock Available at \url{https://rdlyons.pages.iu.edu/}.

\bibitem{JPM23}
Joseph~Paul MacManus.
\newblock Accessibility, planar graphs, and quasi-isometries.
\newblock 2023.
\newblock \url{https://arxiv.org/pdf/2310.15242}.

\bibitem{NewmanSchulman1981}
Charles~M. Newman and Leonard~S. Schulman.
\newblock Infinite clusters in percolation models.
\newblock {\em Journal of Statistical Physics}, 26:613--628, 1981.

\bibitem{Tang2023}
Pengfei Tang.
\newblock A note on some critical thresholds of bernoulli percolation.
\newblock {\em Electronic Journal of Probability}, 28:1--22, 2023.

\end{thebibliography}

\end{document}